\documentclass[12pt,reqno]{amsart}
\usepackage{amssymb}
\usepackage{amsfonts}
\usepackage{amsmath}
\usepackage{amsthm}
\usepackage[hidelinks]{hyperref}
\usepackage{times}
\usepackage{latexsym}
\usepackage[mathscr]{eucal}

\DeclareMathOperator{\id}{id}
\DeclareMathOperator{\Ric}{Ric}

\DeclareMathOperator{\Trace}{Trace}
\DeclareMathOperator{\image}{image}
\DeclareMathOperator{\spann}{span}
\theoremstyle{plain}

\newtheorem{algorithm}{Algorithm}[section]
\newtheorem{thm}{Thm}

\newtheorem{corollary}[algorithm]{Corollary}

\newtheorem{definition}[algorithm]{Definition}
\newtheorem*{Remark-delta}{Remark on all things $\protect\delta$}

\newtheorem{lemma}[algorithm]{Lemma}

\newtheorem{theorem} [algorithm] {Theorem}
\newtheorem{theoremlet}[thm]{Theorem}
\newtheorem{examplelet}[thm]{Example}
\newtheorem{corollarylet}[thm]{Corollary}
\newtheorem{lemmalet}[thm]{Lemma}

\newtheorem*{definitionnonum}{Definition}

\newtheorem{proposition}[algorithm]{Proposition}

\theoremstyle{definition}
\newtheorem*{acknowledgment}{Acknowledgment}
\newtheorem{example}[algorithm]{Example}
\newtheorem{remark}[algorithm]{Remark}

\newtheorem{remarklet}[thm]{Remark}
\begin{document}
\title[Focal Radius, Rigidity, and Lower Curvature Bounds]{Focal Radius,
Rigidity, and Lower Curvature Bounds}
\author{Luis Guijarro}
\address{Department of Mathematics, Universidad Aut\'{o}noma de Madrid, and
ICMAT CSIC-UAM-UCM-UC3M }
\email{luis.guijarro@uam.es}
\urladdr{http://www.uam.es/luis.guijarro}
\thanks{The first author was supported by research grants MTM2011-22612,
MTM2014-57769-3-P from the MINECO, and by ICMAT Severo Ochoa project
SEV-2015-0554 (MINECO)}
\author{Frederick Wilhelm}
\address{Department of Mathematics\\
University of California\\
Riverside, CA 92521}
\email{fred@math.ucr.edu}
\urladdr{https://sites.google.com/site/frederickhwilhelmjr/home}
\thanks{This work was supported by a grant from the Simons Foundation
(\#358068, Frederick Wilhelm)}
\date{May 18, 2016}
\subjclass[2000]{Primary 53C20}
\keywords{Focal Radius, Rigidity, Projective Space, Positive Curvature}

\begin{abstract}
We prove a new comparison lemma for Jacobi fields that exploits Wilking's
transverse Jacobi equation. In contrast to standard Riccati and Jacobi
comparison theorems, there are situations when our technique can be applied
after the first conjugate point.

Using it we show that the focal radius of any submanifold $N$ of positive
dimension in a manifold $M$ with sectional curvature greater than or equal
to $1$ does not exceed $\frac{\pi }{2}.$ In the case of equality, we show
that $N$ is totally geodesic in $M$ and the universal cover of $M$ is
isometric to a sphere or a projective space with their standard metrics,
provided $N$ is closed.

Our results also hold for $k^{th}$--intermediate Ricci curvature, provided
the submanifold has dimension $\geq k.$ Thus in a manifold with Ricci
curvature $\geq n-1,$ all hypersurfaces have focal radius $\leq \frac{\pi }{2%
},$ and space forms are the only such manifolds where equality can occur, if
the submanifold is closed.

Example \ref{hopf holn} and Remark \ref{not Berger} show that our results
cannot be proven using standard Riccati or Jacobi comparison techniques.
\end{abstract}

\maketitle

\pdfbookmark[1]{Introduction}{Introduction}

A Riemannian manifold $M$ has $k^{th}$--intermediate Ricci curvature $\geq
\ell$ if for any orthonormal $\left( k+1\right) $--frame $\left\{
v,w_{1},w_{2},\ldots ,w_{k}\right\} ,$ the sectional curvature sum, $\Sigma
_{i=1}^{k}\mathrm{sec}\left( v,w_{i}\right) ,$ is $\geq \ell$ (\cite{Wu}, 
\cite{Shen}). For brevity we write $\Ric_{k}\,M\geq \ell.$ Motivated by
Myers theorem we show that if $\Ric_{k}M\geq k,$ then all submanifolds with
dimension $\geq k$ have focal radius $\leq \frac{\pi }{2}.$

\begin{theoremlet}
\label{Intermeadiate Ricci Thm}Let $M$ be a complete Riemannian $n$%
--manifold with $\Ric_{k}\geq k$ and $N$ be any submanifold of $M$ with $%
\dim \left( N\right) \geq k.\vspace{0.1in}$

\noindent 1. Every unit speed geodesic $\gamma $ that leaves $N$
orthogonally at time $0$ has at least\linebreak\ $\dim \left( N\right) -k+1$
focal points for $N$ in $\left[ -\frac{\pi }{2},\frac{\pi }{2}\right] ,$
counting multiplicities$.$ In particular, the focal radius of $N$ is $\leq 
\frac{\pi }{2}.\vspace{0.1in}$

\noindent 2. If the focal radius of $N$ is $\frac{\pi }{2},$ then $N$ is
totally geodesic.
\end{theoremlet}

Since $\Ric_{1}M\geq \ell$ means that all sectional curvatures of $M$ are $%
\geq $ $\ell$ and $\Ric_{n-1}M\geq \ell$ means that $M$ has Ricci curvature $%
\geq \ell,$ the theorem applies to $N\subset M$ if either the Ricci
curvature of $M$ is $\geq n-1$ and $N$ is a hypersurface, or the sectional
curvature of $M$ is $\geq 1$ and $\dim \left( N\right) \geq 1.$

We emphasize that $N$ need not be closed or even complete, and there is no
hypothesis about its second fundamental form$.$ On the other hand, if $N$
happens to be closed and have focal radius $\frac{\pi }{2},$ then we
determine $M$ up to isometry.

\begin{theoremlet}
\label{inter Ricci Rigidity thm}Let $M$ be a complete Riemannian $n$%
--manifold with $\Ric_{k}\geq k.$ If $M$ contains a closed, embedded,
submanifold $N$ with focal radius $\frac{\pi }{2}$ and $\dim \left( N\right)
\geq k,$ then $N$ is totally geodesic in $M$, and the universal cover of $M$
is isometric to the sphere or a projective space with the standard metrics.
\end{theoremlet}

In Section \ref{optimal examples for theorems A and B}, we provide examples
showing that the hypothesis on the dimension of $N$ can not be dropped from
either Theorem \ref{Intermeadiate Ricci Thm} or \ref{inter Ricci Rigidity
thm}.

In the course of proving Theorem \ref{inter Ricci Rigidity thm} we will also
establish the following corollary (see Theorem \ref{const curv thm}, below.)

\begin{corollarylet}
\label{hyper refin cor}If the submanifold $N$ of Theorem \ref{inter Ricci
Rigidity thm} is a hypersurface, then the universal cover of $M$ is
isometric to the unit sphere.
\end{corollarylet}

It is reasonable to compare the Ricci curvature versions of Theorems \ref%
{Intermeadiate Ricci Thm} and \ref{inter Ricci Rigidity thm} with the
Bonnet-Myers Theorem and Cheng's Maximal Diameter Theorem (cf also Theorem 3
in \cite{CrKl} and Theorem 1 in \cite{EschOSu}). While an analogy can be
made between the sectional curvature version of Theorem \ref{inter Ricci
Rigidity thm} and the Diameter Rigidity Theorem (\cite{GrovGrom},\cite{Wilk}%
), the following example shows that Theorem \ref{inter Ricci Rigidity thm}
applies to more nonsimply connected manifolds.

\begin{examplelet}
\label{quat exam}Let $\mathbb{S}^{3}$ be the unit sphere in $\mathbb{C}%
\oplus \mathbb{C},$ and embed $\mathbb{S}^{1}$ as the unit circle in the
first copy of $\mathbb{C}.$ Let $Q$ be the quaternion group of order $8$ in $%
SO\left( 4\right) .$ Then the focal radius of $N=Q\left( \mathbb{S}%
^{1}\right) /Q$ in $M=\mathbb{S}^{3}/Q$ is $\frac{\pi }{2},$ and $N$ is its
own focal set. On the other hand, $M$ has diameter strictly smaller than $%
\frac{\pi }{2}$.

More generally, let $\pi :\mathbb{S}^{n}\longrightarrow \mathbb{S}^{n}/G$ be
the quotient map of a properly discontinuous action by $G$ on $\mathbb{S}%
^{n},$ and let $N$ be any closed geodesic in $\mathbb{S}^{n}/G.$ Then $\pi
^{-1}\left( N\right) $ is the disjoint union of closed geodesics in $\mathbb{%
S}^{n},$ and hence both $\pi ^{-1}\left( N\right) $ and $N$ have focal
radius $\frac{\pi }{2}.$
\end{examplelet}

Theorem \ref{inter Ricci Rigidity thm} implies that the standard unit metric
is the only one on any topological sphere with sectional curvature $\geq 1$
that has a closed submanifold with focal radius $\frac{\pi }{2}.$ In
contrast, the conclusion of the Diameter Rigidity Theorem is softer, since
there are many metrics on $\mathbb{S}^{n}$ with curvature $\geq 1$ and
diameter $\geq \frac{\pi }{2},$ and there is even the possibility of such a
metric on an exotic sphere.

It is also reasonable to compare the sectional curvature version of Theorem %
\ref{inter Ricci Rigidity thm} to the \textquotedblleft rank
rigidity\textquotedblright\ results of Schmidt and
Shankar--Spatzier--Wilking in \cite{Schm} and \cite{ShankSpWilk}. Shankar,
Spatzier, and Wilking obtained the conclusion of Theorem \ref{inter Ricci
Rigidity thm} for manifolds with curvature \emph{less }than or equal to $1$
and minimal conjugate radius $\pi .$ Schmidt proves that if $M$ has
sectional curvature $\geq 1$ and conjugate radius $\geq \frac{\pi }{2},$
then its universal cover is homeomorphic to $S^{n}$ or isometric to a
projective space. The conjugate radius hypotheses of these theorems apply to
every geodesic in $M.$ In contrast, the focal radius hypothesis of Theorem %
\ref{inter Ricci Rigidity thm} only concerns the geodesics that meet a
single submanifold orthogonally.

To prove Theorems \ref{Intermeadiate Ricci Thm} and \ref{inter Ricci
Rigidity thm}, we exploit Wilking's transverse Jacobi equation (\cite{Wilk1}%
) to get a new comparison lemma for Jacobi fields. To state it, we let $%
\gamma :\left( -\infty ,\infty \right) \longrightarrow M$ be a unit speed
geodesic in a complete Riemannian $n$--manifold $M.$ We call an $\left(
n-1\right) $--dimensional subspace $\Lambda $ of normal Jacobi fields along $%
\gamma ,$ \emph{Lagrangian, }if the restriction of the Riccati operator to $%
\Lambda $ is self adjoint, that is, if 
\begin{equation*}
\left\langle J_{1}\left( t\right) ,J_{2}^{\prime }\left( t\right)
\right\rangle =\left\langle J_{1}^{\prime }\left( t\right) ,J_{2}\left(
t\right) \right\rangle
\end{equation*}%
for all $t$ and for all $J_{1},J_{2}\in \Lambda $ (see \eqref{Rica dfn}
below for the formal definition of the Riccati operator on $\Lambda $)$.$

In Sections \ref{sect:Wilking equation} and \ref{Jacobi-comp sect}, we
review Wilking's transverse Jacobi equation, justify the name Lagrangian,
and prove a comparison lemma for intermediate Ricci curvature. In the
special case when the sectional curvature is bounded from below our
comparison result becomes the following.

\begin{lemmalet}
\label{sing soon thm}(Sectional Curvature Comparison) For $\kappa =-1,0,$ or 
$1,$ let $\gamma :\left( -\infty ,\infty \right) \longrightarrow M$ be a
unit speed geodesic in a complete Riemannian $n$--manifold $M$ with \textrm{%
sec}$\left( \dot{\gamma},\cdot \right) \geq \kappa .$ Let $J_{0}$ be a
nonzero, normal Jacobi field along $\gamma ,$ and let $\Lambda $ be a
Lagrangian subspace of normal Jacobi fields along $\gamma $ with Riccati
operator $S$ such that $J_{0}\in \Lambda .$

For $t_{0}<t_{\max },$ suppose that $\Lambda $ has no singularities on $%
\left( t_{0},t_{\max }\right) ,$ and that $\tilde{\lambda}_{\kappa }:\left[
t_{0},t_{\max }\right) \longrightarrow \mathbb{R}$ is a solution of %
\addtocounter{algorithm}{1} 
\begin{equation}
\tilde{\lambda}_{\kappa }^{\prime }+\tilde{\lambda}_{\kappa }^{2}+\kappa =0
\label{const curv ODE}
\end{equation}%
with \addtocounter{algorithm}{1} 
\begin{equation}
\left\langle S\left( J_{0}\right) ,J_{0}\right\rangle |_{t_{0}}\leq \tilde{%
\lambda}_{\kappa }\left( t_{0}\right) \left\vert J_{0}\left( t_{0}\right)
\right\vert ^{2}.  \label{sec iniitial cond}
\end{equation}%
Then for each $t_{1}\in \left[ t_{0},t_{\max }\right) $ there is a $J_{1}\in
\Lambda \setminus \left\{ 0\right\} $ so that\addtocounter{algorithm}{1} 
\begin{equation}
\left\langle S\left( J_{1}\right) ,J_{1}\right\rangle |_{t_{1}}\leq \tilde{%
\lambda}_{\kappa }\left( t_{1}\right) \left\vert J_{1}\left( t_{1}\right)
\right\vert ^{2}.  \label{sm fut sec Ineq}
\end{equation}%
In particular, if $\kappa =1$, $\alpha \in \left[ 0,\pi \right) ,$ $\tilde{%
\lambda}_{1}\left( t\right) =\cot \left( t+\alpha \right) $, and $t_{0}\in %
\left[ 0,\pi -\alpha \right) ,$ then $\Lambda $ has a singularity by time $%
\pi -\alpha ,$ that is, there is a $J\in \Lambda \setminus \left\{ 0\right\} 
$ with $J\left( t_{2}\right) =0$ for some $t_{2}\in \left( t_{0},\pi -\alpha %
\right] .$
\end{lemmalet}

%\begin{remarknonum}
Lemma \ref{sing soon thm} holds in certain situations where $\Lambda $ has
singularities on $\left[ t_{0},t_{\max }\right) ,$ for example when $%
\lim_{t\rightarrow t_{0}^{+}}\tilde{\lambda}_{\kappa }\left( t\right)
=\infty .$ We describe another such situation in Lemma \ref{sing soon Ric_k
Lemma}, where the reader will also find a discussion of the equality case.

The reader is probably familiar with the Riccati comparison theorem of
Eschenburg-Heintze in \cite{EschHein}. It requires the initial condition (%
\ref{sec iniitial cond}) to hold for\emph{\ all} $J_{0}\in \Lambda $, while
Lemma \ref{sing soon thm} only demands that the initial condition holds for
a \emph{single} Jacobi field. This comes at the expense that the derived
future inequality (\ref{sm fut sec Ineq}) is only guaranteed to hold for a
single Jacobi field, which moreover, is not likely to be the original field.
In Examples \ref{diff future ex} and \ref{hopf holn} (below), we show that $%
J_{1}$ can in fact be different from $J_{0}.$ A similar example can be found
on page 463 of \cite{HeiKar}. This phenomenon is tied to the nonvanishing of
Wilking's generalized $A$--tensor (see (\ref{Burk A dfn})).

The difference between Lemma \ref{sing soon thm} and the theorem of \cite%
{EschHein} is starker if one considers the contrapositives: Lemma \ref{sing
soon thm} implies that if Inequality \eqref{sm fut sec Ineq} fails for all $%
J_{1}\in \Lambda ,$ then Inequality \eqref{sec iniitial cond} fails for 
\emph{all} $J\in \Lambda .$ In contrast, the theorem of \cite{EschHein} only
gives that Inequality \ref{sec iniitial cond} fails for \emph{some} $J\in
\Lambda .$

The main tool to prove Theorem \ref{Intermeadiate Ricci Thm} is Lemma \ref%
{sing soon Ric_k Lemma}, which is a generalization of Lemma \ref{sing soon
thm} to intermediate Ricci curvature. So that we can prove Theorem \ref%
{inter Ricci Rigidity thm}, Lemma \ref{sing soon Ric_k Lemma} also includes
an analysis of the rigid situation. Other cases when rigidity occurs are
given in Lemmas \ref{lem:riccati comparison minus infinity} and \ref%
{lem:riccati comparison nnc} % For the convenience of the
%reader we detail in Corollary \ref{Sec rigid cor} the rigidity case of Lemma %
%\ref{sing soon thm}, that is we explain what happens when Inequality 
%\eqref{sm fut sec
%Ineq} is an equality.

The proof of Theorem \ref{inter Ricci Rigidity thm} begins by establishing
Proposition \ref{Jacobi Rigidity prop}, which draws a strong analogy between 
$N$ and one of the dual sets in the proof of the Diameter Rigidity Theorem.
Example \ref{quat exam} shows that we can only push this analogy so far. The
dual sets of \cite{GrovGrom} are disjoint while Example \ref{quat exam}
shows that $N$ can be its own focal set. In fact, one of the challenges of
the proof of Theorem \ref{inter Ricci Rigidity thm} is showing that
phenomena like Example \ref{quat exam} do not occur in the simply connected
case. In spite of the differences, our overall strategy is similar to that
of \cite{GrovGrom}, and our proof employs ideas from there. To keep the
exposition tight, we will often refrain from giving further specific
references to \cite{GrovGrom} and have made our exposition reasonably
self-contained.

After the introduction, we establish notations and conventions. The
remainder of the paper is divided into two parts and eight sections. The
sections are subordinate to the parts. Each part and many of the sections
begin with a detailed summary of the contents, so the outline immediately
below is only meant to indicate where each result is proven.

Part 1 %\ref{Part 1} 
contains Sections \ref{sect:Wilking equation} to \ref{focal and pos sect}.
In Section \ref{sect:Wilking equation}, we review Wilking's transverse
Jacobi equation; in Section \ref{Jacobi-comp sect} we state and prove Lemma %
\ref{sing soon Ric_k Lemma}, which is the main tool of the paper. Subsection %
\ref{example seubsection} provides examples showing that $J_{0}$ and $J_{1}$
can indeed be different in Lemma \ref{sing soon thm}. In Section \ref{focal
and pos sect}, we prove Theorem \ref{Intermeadiate Ricci Thm} and give
examples showing its optimality.

In Part 2, %\ref{focal rig Part}, 
we prove Theorem \ref{inter Ricci Rigidity thm} in Sections \ref{dist from N
sect}---\ref{Rigid and Inter ric sect}. In the special case of Theorem \ref%
{inter Ricci Rigidity thm}, when the sectional curvature is $\geq 1$, the
argument can be completed a little faster by an appeal to the Diameter
Rigidity Theorem. We do this in Section \ref{M is not a sphere}, and we
complete the proof of the general case of Theorem \ref{inter Ricci Rigidity
thm} in Section \ref{Rigid and Inter ric sect}.

\begin{remarklet}
\label{radial rem}The reader may have noticed that the hypotheses $\Ric_{k}$ 
$\geq k\cdot \kappa $ of Theorems \ref{Intermeadiate Ricci Thm} and \ref%
{inter Ricci Rigidity thm} are global, whereas in Lemma \ref{sing soon thm},
we only assumed that \textrm{sec}$\left( \dot{\gamma},\cdot \right) \geq
\kappa .$

For the conclusion of Theorems \ref{Intermeadiate Ricci Thm} to hold, we in
fact, only need $\Ric_{k}$ $\left( \dot{\gamma},\cdot \right) \geq k\cdot
\kappa $ for all unit speed geodesics $\gamma $ that leave $N$ orthogonally
at time zero. That is,%
\begin{equation*}
\displaystyle\sum\limits_{i=1}^{k}\mathrm{sec}\left( \dot{\gamma}%
,E_{i}\right) \geq k\cdot \kappa
\end{equation*}%
for any orthonormal set $\left\{ \dot{\gamma},E_{1},\ldots ,E_{k}\right\} .$

On the other hand, our proof of Theorem \ref{inter Ricci Rigidity thm} uses
the global hypothesis $\Ric_{k}$ $\geq k\cdot \kappa $ and also the fact
that Lemma \ref{sing soon Ric_k Lemma} and its rigidity case are valid with
only the radial curvature lower bound.
\end{remarklet}

%\begin{remarknonum}
%Although the Ricci curvature version of Theorem \ref{Intermeadiate Ricci Thm}
%can be proven via standard Riccati comparison (see e.g. \cite{EschHein}),
%its statement does not seem to be in the literature. In contrast, it does
%not seem possible to prove the sectional version of Theorem \ref%
%{Intermeadiate Ricci Thm} with existing Jacobi or Riccati comparison
%results, see Example \ref{hopf holn} and Remark \ref{not Berger}.
%\end{remarknonum}

\begin{acknowledgment}
We are grateful to Karsten Grove and Curtis Pro for valuable critiques of
this manuscript. Special thanks go to Universidad Autonoma de Madrid for
hosting a stay by the second author during which this work was initiated.
\end{acknowledgment}

\section*{Notations and Conventions \label{notations Sections}}

Unless otherwise specified, all curves are parameterized at unit speed.
Given $v\in TM,$ we denote the unique geodesic with $\gamma _{v}^{\prime
}\left( 0\right) =v$ \/ by $\gamma _{v}.$

Let $N$ be a submanifold of the Riemannian manifold $M.$ Let $\nu \left(
N\right) $ be the normal bundle of $N\subset M.$ For every unit $v\in \nu
\left( N\right) ,$ there is a first time $t_{1}\in \left( 0,\infty \right] $
at which $\gamma _{v}\left( t_{1}\right) $ is focal for $N$ along $\gamma
_{v}.$ We set 
\begin{equation*}
\mathrm{reg}_{N}\equiv \left\{ \left. tv\in \nu \left( N\right) \text{ }%
\right\vert \text{ }\left\vert v\right\vert =1\text{ and }t\in \left[
0,t_{1}\right) \right\} .
\end{equation*}%
We let $g^{\ast }$ be the metric on the domain $\mathrm{reg}_{N}$ obtained
from pulling back $\left( M,g\right) $ via the normal exponential map. We
use the term \emph{tangent focal point} for a critical point of $\exp
_{N}^{\perp }:\nu \left( N\right) \longrightarrow M$ and the term \emph{%
focal point} for a critical value of $\exp _{N}^{\perp }.$

$\pi :\nu \left( N\right) \longrightarrow N$ will denote the projection of
the normal bundle; $N_{0}$ will be the $0$--section of $\nu \left( N\right)
, $ and $\nu ^{1}\left( N\right) $ will be the unit normal bundle of $N.$
The fibers of $\nu \left( N\right) $ and $\nu ^{1}\left( N\right) $ over $%
x\in N$ will be called $\nu _{x}\left( N\right) $ and $\nu _{x}^{1}\left(
N\right) $.

We let $\Lambda $ be any Lagrangian family of normal Jacobi fields along a
geodesic $\gamma ,$ and for any subspace $\mathcal{W}\subset \Lambda $ we
write \addtocounter{algorithm}{1} 
\begin{equation}
\mathcal{W}\left( t\right) \equiv \left\{ \left. J\left( t\right) \text{ }%
\right\vert \text{ }J\in \mathcal{W}\right\} \oplus \left\{ \left. J^{\prime
}\left( t\right) \text{ }\right\vert \text{ }J\in \mathcal{W}\text{ and }%
J\left( t\right) =0\right\} .  \label{dfn of eval eqn}
\end{equation}%
When $\gamma $ is a geodesic that leaves $N$ orthogonally at time $0$, we
will write $\Lambda _{N}$ for the Lagrangian family of normal Jacobi fields
along $\gamma $ corresponding to variations by geodesics that leave $N$
orthogonally at time $0$. We call the elements of $\Lambda _{N},$ $N$%
--Jacobi fields. According to Lemma 4.1 on page 227 of \cite{doCarm}, $%
\Lambda _{N}$ consists of the following \emph{normal} Jacobi fields $J$
along $\gamma $: \addtocounter{algorithm}{1} 
\begin{equation}
\Lambda _{N}\equiv \left\{ J|J\left( 0\right) =0\text{, }J^{\prime }\left(
0\right) \in \nu _{\gamma \left( 0\right) }\left( N\right) \right\} \oplus
\left\{ J|J\left( 0\right) \in T_{\gamma \left( 0\right) }N\text{ and }%
J^{\prime }\left( 0\right) =\mathrm{S}_{\gamma ^{\prime }\left( 0\right)
}J\left( 0\right) \right\} ,  \label{dfn of Lambda_N}
\end{equation}%
where $\mathrm{S}_{\gamma ^{\prime }\left( 0\right) }$ is the shape operator
of $N$ determined by $\gamma ^{\prime }\left( 0\right) ,$ that is, 
\begin{eqnarray*}
\mathrm{S}_{\gamma ^{\prime }\left( 0\right) } &:&T_{\gamma _{v}\left(
0\right) }N\longrightarrow T_{\gamma _{v}\left( 0\right) }N\text{ is} \\
\mathrm{S}_{\gamma ^{\prime }\left( 0\right) } &:&w\longmapsto \left( \nabla
_{w}\gamma ^{\prime }\left( 0\right) \right) ^{TN}.
\end{eqnarray*}

We write $\mathbb{S}^{n}$ for the unit sphere in $\mathbb{R}^{n+1},$ and for 
$\kappa =-1,0,$ or $1,$ we let $\mathcal{S}_{\kappa }^{2}$ be the simply
connected $2$--dimensional space form of constant curvature $\kappa .$

We use the acronym CROSS for Compact Rank One Symmetric Space. For
convenience, we normalize the nonspherical CROSSes so that their curvatures
are in $\left[ 1,4\right] ,$ and we normalize the spherical CROSSes to have
constant curvature $4.$

We write $\mathrm{sec}$ for sectional curvature and $\kappa $ for our lower
curvature bound. After rescaling, we may always assume that $\kappa $ is
either $-1,0,$ or $1.$

Given $r>0$ and $A\subset M$ we set 
\begin{eqnarray*}
B\left( A,r\right) &\equiv &\left\{ \left. x\in M\text{ }\right\vert \text{ 
\textrm{dist}}\left( x,A\right) <r\right\} ,\text{ } \\
D\left( A,r\right) &\equiv &\left\{ \left. x\in M\text{ }\right\vert \text{ 
\textrm{dist}}\left( x,A\right) \leq r\right\} ,\text{ and} \\
S\left( A,r\right) &\equiv &\left\{ \left. x\in M\text{ }\right\vert \text{ 
\textrm{dist}}\left( x,A\right) =r\right\} .
\end{eqnarray*}

Finally, we write $D_{v}\left( f\right) $ the derivative of $f$ in the
direction $v.$

\part*{{\protect\Large Part 1: Bounding the Focal Radius\label{Part 1}}}

Part 1 is divided in three sections. Section \ref{sect:Wilking equation}
reviews Wilking's transverse equation. In Section \ref{Jacobi-comp sect}, we
state and prove Lemma \ref{sing soon Ric_k Lemma}, which is a generalization
of Lemma \ref{sing soon thm} and is the main tool of the paper; in
subsection \ref{example seubsection} we give an example that shows that $%
J_{1}$ need not equal $J_{0} $ in Lemma \ref{sing soon thm}. Finally, in
Section 3, we prove Theorem \ref{Intermeadiate Ricci Thm}, and give some
examples showing its optimality.

%\section{Jacobi Field Comparison\label{Jacobi-comp sect}}

\section{Wilking's transverse Jacobi equation}

%\label{Jacobi-comp sect}}
\label{sect:Wilking equation}

\renewcommand{\theequation}{\thealgorithm}

In this section, we review Lagrangian families and Wilking's transverse
Jacobi equation.

\subsection{Lagrangian Families}

Let $\gamma $ be a unit speed geodesic in a complete Riemannian $n$%
--manifold $M,$ and let $\mathcal{J}$ be the vector space of normal Jacobi
fields along $\gamma .$ Using symmetries of the curvature tensor, we see
that for $J_{1},J_{2}\in \mathcal{J},$ 
\begin{equation*}
\omega \left( J_{1},J_{2}\right) =\left\langle J_{1}^{\prime
},J_{2}\right\rangle -\left\langle J_{1},J_{2}^{\prime }\right\rangle ,
\end{equation*}%
is constant along $\gamma $ and hence defines a symplectic form on $\mathcal{%
J}.$

Thus an $\left( n-1\right) $--dimensional subspace $\Lambda $ of $\mathcal{J}
$ on which $\omega $ vanishes is called Lagrangian. Of course this is
equivalent to saying that the restriction of the Riccati operator to $%
\Lambda $ is self-adjoint. Examples of Lagrangian families include the
Jacobi fields that are $0$ at time $0$ and those that correspond to
variations by geodesics that leave a submanifold orthogonally at time $0.$

The set of times $t$ so that

\addtocounter{algorithm}{1}

\begin{equation}
\left\{ J(t)\ |\ J\in \Lambda \right\} = \dot{\gamma}\left( t\right)^{\perp }
\label{Eval 1-1 eqn}
\end{equation}%
is open and dense (cf Lemma 1.7 of \cite{gw}). For these $t$ we get a
well-defined Riccati operator \addtocounter{algorithm}{1} 
\begin{eqnarray}
S_{t} &:&\dot{\gamma}\left( t\right)^{\perp }\longrightarrow \dot{\gamma}%
\left( t\right)^{\perp }  \notag \\
S_{t} &:&v\longmapsto J_{v}^{\prime }\left( t\right) ,  \label{Rica dfn}
\end{eqnarray}%
where $J_{v}$ is the unique $J_{v}\in \Lambda $ so that $J_{v}\left(
t\right) =v.$ The Jacobi equation then decomposes into the two first order
equations 
\begin{equation*}
S_t\left( J\right) =J^{\prime },\qquad \text{ }S_t^{\prime }+S_t^{2}+R=0,
\end{equation*}%
where $S_t^{\prime }$ is the covariant derivative of $S_t$ along $\gamma $
and $R $ is the curvature along $\gamma ,$ that is $R\left( \cdot \right)
=R\left( \cdot ,\dot{\gamma}\right) \dot{\gamma}$ (see Equation 1.7.1 in 
\cite{gw}). We will omit the dependence on $t$ if it is clear from the
context.

\begin{remark}
\label{S well dfned on W rem} Given any $\mathcal{W}\subset \Lambda $, and
some $t$ such that no Jacobi field in $\mathcal{W}\setminus \left\{
0\right\} $ vanishes at $t$, Equation \eqref{Rica dfn} gives a well defined
Riccati operator 
\begin{equation*}
S_{t}:\mathcal{W}\left( t\right) \longrightarrow \gamma ^{\prime }\left(
t\right) ^{\perp }.
\end{equation*}%
This $S_{t}$ agrees with the restriction of $S_{t}$ defined in 
\eqref{Rica
dfn} when $\Lambda $ has no zeros.

% 
%It is clear that the condition on $\mathcal{W}$ and $t$ is equivalent to  
%$\mathcal{E}%
%_{t}|_{\mathcal{W}}$ being one-to-one.
\end{remark}

\bigskip

%\subsection{Jacobi operators}
%By choosing an orthonormal parallel frame normal to the geodesic $\gamma:I\to M$, we can identify each subspace $\gamma'(t)^\perp$ with a fixed vector space $E=\gamma'(t_0)^\perp$.  
%A Lagrangian of Jacobi fields along a geodesic induces linear maps $\calJ(t):E\to E$
%as
%\[
%\calJ(t)(v)=P_{\gamma,t}^{-1}(J(t))
%\]
%where $J$ is the unique Jacobi field in $\Lambda$ with $J(t_0)=v$, and $P_{\gamma,t}$ is the parallel transport along $\gamma$ from $t_0$ to $t$.
%
%We will say that the Lagrangian $\Lambda$ splits along $\gamma$ if there is an orthogonal direct sum $E=E_1\oplus E_2$ such that the map $\calJ$ splits as $\calJ=\calJ_1\oplus\calJ_2$, where $\calJ_i:E_i\to E_i$, $i=1,2$. 
%
%\begin{lemma}
%Suppose that the Lagrangian $\Gamma$ splits along $\gamma$. Then
%\begin{enumerate}
%\item there is a basis of Jacobi fields in $\Lambda$, $\{X_1,\dots,X_k, Y_1, \dots Y_{\ell}\}$, such that $X_i\perp Y_j$, $X_i'\perp Y_j$, $X_i\perp Y_j'$ for every possible pair $i,j$;
%\item the vector subbundles $E_i(t):=P_{\gamma(t)}E_i$ are parallel along $\gamma$;
%\item the curvature endomorphism $R(t):=R(\gamma'(t), \cdot)\gamma'(t)$ splits as $R(t)=R_1(t)\oplus R_2(t)$ with $R_i(t)E_i(t)\subset E_i(t)$.
%\end{enumerate}
%\end{lemma}
%
%\bigskip

\subsection{Singularities in the Lagrangian and the Riccati operator}

The set of times $t$ when \addtocounter{algorithm}{1} 
\begin{equation}
\dim\left\{ J(t)\ |\ J\in \Lambda \right\}< n-1  \label{Eval not 1-1 eqn}
\end{equation}
corresponds to the moments where some of the Jacobi fields in $\Lambda$
vanish. They are important since, in general, they correspond to moments
when the Riccati operator $S_t$ is not defined.

\begin{definitionnonum}
Let $\mathcal{V}$ be a subspace of $\Lambda $. We will say that $\mathcal{V}$
has full index at $\bar{t}$ if any $J\in \Lambda $ with $J(\bar{t})=0\/$
belongs to $\mathcal{V}$; we will also say that $\mathcal{V}$ has full index
on an interval $I$ if it has full index at each point of $I$.
\end{definitionnonum}

%\begin{remark}
%\label{S well dfned on W rem}
There is a different way of stating the above condition: for fixed $t\in 
\mathbb{R},$ define the \emph{evaluation map} as%
\begin{eqnarray*}
\mathcal{E}_{t} &:&\Lambda \longrightarrow T_{\gamma \left( t\right) }M \\
\mathcal{E}_{t} &:&J\longmapsto J\left( t\right).
\end{eqnarray*}

Observe that, for given $t\in I$, the kernel of $\mathcal{E}_{t}$ is the set
of those $J\in \Lambda $ vanishing at $t$. Thus a subspace $\mathcal{V}%
\subset \Lambda $ has full index in an interval if and only if $\mathcal{V}$
contains the kernel of the evaluation map $\mathcal{E}_{t}$ for every $t$ in
the interval.

\bigskip

\subsection{Wilking's Transverse Jacobi Equation}

Let $\mathcal{V}$ be any subspace of $\Lambda .$ Set %
\addtocounter{algorithm}{1} 
\begin{equation}
\mathcal{V}(t)\equiv \{J(t)\ |\ J\in \mathcal{V}\}\oplus \{J^{\prime }(t)\
|\ J\in \mathcal{V},J(t)=0\}.  \label{dfn of V(t)}
\end{equation}%
Then $\mathcal{V}(t)$ is a smooth vector bundle along $\gamma $ (Lemma 1.7.1
in \cite{gw}, or \cite{Wilk1}). Set 
\begin{equation*}
H(t)\equiv \mathcal{V}(t)^{\perp }\cap \dot{\gamma}(t)^{\perp }.
\end{equation*}

\begin{proposition}
\label{S hat well defined prop} Fix $t\in I$ and suppose that $\mathcal{V}$
has full index at $t.\vspace{0.1in}$

\noindent 1. For $x\in H\left( t\right) ,$ there is a $J\in \Lambda $ so
that $J\left( t\right) =x.\vspace{0.1in}$

\noindent 2. We have a well-defined Riccati operator%
\begin{equation*}
\hat{S}_{t}:H\left( t\right) \longrightarrow H\left( t\right)
\end{equation*}%
given by \addtocounter{algorithm}{1} 
\begin{equation}
\hat{S}_{t}\left( x\right) =J^{\prime H}\left( t\right) ,\text{ \label{dfn
of S hat}}
\end{equation}%
where $J$ is an element of $\Lambda $ so that $J\left( t\right) =x,$ and $%
J^{\prime H}\left( t\right) $ is the $H\left( t\right) $--component of $%
J^{\prime }\left( t\right) $; in other words, $\hat{S}_{t}$ is the $H(t)$%
-projection of $S_{t}|_{H\left( t\right) }$.
\end{proposition}

\begin{proof}
Since $\Lambda $ is Lagrangian, the splitting 
\begin{equation*}
\Lambda \left( t\right) =\{J(t)\ |\ J\in \Lambda \}\oplus \{J^{\prime }(t)\
|\ J\in \Lambda ,J(t)=0\}
\end{equation*}%
is orthogonal. Since the kernel of $\mathcal{E}_{t}$ lies in $\mathcal{V},$ $%
H\left( t\right) $ is contained in the first summand, and Part 1 follows.

%For the second part, suppose $x=J\left( t\right) +K\left( t\right) ,$ where $%
%K\in \mathrm{Kernel}\left( \mathcal{E}_{t}\right) .$ Since $\mathrm{Kernel}%
%\left( \mathcal{E}_{t}\right) \subset \mathcal{V},$ 
%\begin{equation*}
%\left( J+K\right) ^{H}=J^H,
%%\left. \left( \left( \left( J+K\right) ^{H}\right) ^{\prime }\right)
%%^{H}\right\vert _{t}=\left( \left( J^{H}\right) ^{\prime }\right) ^{H}|_{t},
%\end{equation*}%
%as claimed.

For the second part, suppose $x=J_{1}\left( t\right) =J_{2}\left( t\right)
\in H\left( t\right) $ and $J_{1},J_{2}\in \Lambda .$ Since $J_{1}-J_{2}$
vanishes at $t$ and $\mathrm{Kernel}\left( \mathcal{E}_{t}\right) \subset 
\mathcal{V},$ we have $J_{1}-J_{2}\in \mathcal{V}.$ Together with $\left(
J_{1}-J_{2}\right) \left( t\right) =0,$ this implies that $\left(
J_{1}-J_{2}\right) ^{\prime }\left( t\right) \in \mathcal{V}(t).$ Thus $%
\left( \left( J_{1}-J_{2}\right) ^{\prime }\left( t\right) \right) ^{H}=0,$
and $\hat{S}(x)$ is independent of the choice of $J\in \Lambda $ so that $%
J(t)=x$.
\end{proof}

We will call $\hat{S}$ \emph{the Riccati operator associated to $\mathcal{V}$%
}, if it is clear which Lagrangian $\Lambda $ is being used.

Wilking also defined maps \addtocounter{algorithm}{1} 
\begin{eqnarray}
A_{t} &:&\mathcal{V}\left( t\right) \longrightarrow H\left( t\right) \text{
given by,}  \notag \\
A_{t}\left( v\right) &=&\left( J^{\prime }\right) ^{h}\left( t\right) ,\text{
where }J\in \mathcal{V},\text{ }J\left( t\right) =v.  \label{Burk A dfn}
\end{eqnarray}

A priori, $A_{t}$ is only defined at points where $\Lambda $ has no zeros;
however, $A$ extends smoothly to $\mathbb{R}$ (cf. \cite{Wilk1}). Indeed,
let $A_{t}^{\ast }:H\left( t\right) \longrightarrow \mathcal{V}\left(
t\right) $ be the adjoint of $A_{t},$ and let $X$ be a field in $H$ so that $%
\left( X^{\prime }\right) ^{H}\equiv 0.$ According to Equation 1.7.6 on page
38 of \cite{gw}, 
\begin{equation*}
X^{\prime }=-A^{\ast }X.
\end{equation*}%
Since the left-hand side is smooth, $A^{\ast }$ is smooth, and it follows
that $A$ is smooth.

\begin{theorem}[Wilking \protect\cite{Wilk1}]
\label{HCE thm}$\hat{S}$ is self-adjoint, and \addtocounter{algorithm}{1} 
\begin{equation}
\hat{S}^{\prime }+\hat{S}^{2}+\left\{ R\left( \cdot ,\dot{\gamma}\right) 
\dot{\gamma}\right\} ^{h}+3AA^{\ast }=0.  \label{Horiz Wilki}
\end{equation}
\end{theorem}

Equation \eqref{Horiz Wilki} is known as the Transverse Jacobi Equation. It
is a vast generalization of the Horizontal Curvature Equation of \cite{Gray}
and \cite{O'Ne}. For details see \cite{GumWilh} or \cite{Ly}.

Proposition \ref{S hat well defined prop} only gives us that $\hat{S}$ is
defined almost everywhere. However, $\hat{S}^{\prime }+\hat{S}^{2}$ has a
smooth extension to all of $\mathbb{R}$, because $\left\{ R\left( \cdot ,%
\dot{\gamma}\right) \dot{\gamma}\right\} ^{h}+3AA^{\ast }$ is smooth
everywhere (see \cite{Wilk1} for an interpretation of $\hat{S}^{\prime }+%
\hat{S}^{2}$\/ as a second order differential operator $H\left( t\right)
\longrightarrow H\left( t\right) )$.

\bigskip

\subsection{Splitting of Lagrangians}

Like the Gray-O'Neill $A$--tensor, the Wilking $A$--tensor vanishes
identically along a geodesic $\gamma $ if and only if the distributions $%
\mathcal{V}\left( t\right) $ and $H\left( t\right) $ are parallel along $%
\gamma .$ In this case, it follows that the subspaces of $\Lambda ,$ 
\begin{equation*}
\left\{ \left. J\in \Lambda \text{ }\right\vert \text{ }J\left( t\right) \in
H\left( t\right) \right\} ,
\end{equation*}%
are independent of $t,$ and the parallel, orthogonal splitting $\mathcal{V}%
\left( t\right) \oplus H\left( t\right) $ is given by Jacobi fields. We make
this more rigorous in what follows.

\begin{lemma}
\label{lem:Lagrangian splitting} With the above notation, assume that $A_t=0$
for every $t\in I$. Then

\begin{enumerate}
\item $\mathcal{V}(t)$ and $H(t)$ are parallel distributions along $\gamma $.

\item If for some $\bar{t}\in I$, a Jacobi field $J\in \Lambda $ has $J(\bar{%
t})\in H(\bar{t})$, then $J(t)\in H(t)$ for every $t$.

\item There is a subspace $\mathcal{H}\subset \Lambda $ such that $\mathcal{H%
}(t)=H(t)$ for every $t$. %at every point where $\Lambda$ has not zeros, 
%\[
%\mathcal{V}(t)=V(t), \quad \mathcal{H}(t)=H(t).
%\]

%\item The Jacobi equation splits; i.e, there are subspaces $\mathcal{V}$, $\mathcal{H}$ of $\Lambda$, such that 
%for any $t$ where $\Lambda$ has no zeros, 
%\[
%\mathcal{V}(t)=V(t), \quad 
%\mathcal{H}(t)=H(t).
%\]
%
%
%\item The Jacobi equation splits, i.e., the Jacobi operator induced by $\Lambda$ splits as         $\calJ=\calJ_H\oplus \calJ_V$, where $J_H(t)\in H(t)$, $J_V(t)\in V(t)$ for every $t\in I$.
%\item The subspace of  $\Lambda$ defined as 
%\begin{equation*}
%\left\{ \left. J\in \Lambda \text{ }\right\vert \text{ }J\left( t\right) \in
%H\left( t\right) \right\}
%\end{equation*}%
%is independent of $t$.
\end{enumerate}
\end{lemma}

\begin{proof}
By continuity, it is enough to check the first part at times $t\in I$ where $%
\Lambda $ has no zeros. Since any section of the bundle $\mathcal{V}(t)$ can
be written as 
\begin{equation*}
Y=\sum_{i}f_{i}\cdot J_{i},
\end{equation*}%
where $J_{i}$ are a basis of $\mathcal{V}$, we have that 
\begin{equation*}
Y^{\prime H}=\sum_{i}f_{i}\cdot J_{i}^{\prime H}=0
\end{equation*}%
since $A_{t}\equiv 0$. Therefore $\mathcal{V}(t)$, and consequently $%
\mathcal{V}^{\perp }=\mathcal{H}$, are both parallel, proving the first part
of the Lemma.

Since $\mathcal{V}(t)$ is parallel and spanned by Jacobi fields, it follows
that $R\left( \cdot ,\gamma ^{\prime }\right) \gamma ^{\prime }$ leaves $%
\mathcal{V}(t)$ invariant. From this it follows that $R\left( \cdot ,\gamma
^{\prime }\right) \gamma ^{\prime }$ leaves $H(t)$ invariant. Combining this
with the fact that $H(t)$ is parallel, we get Part 2.

For the last part, choose a set $\{J_{1},\dots J_{\ell }\}$ in $\Lambda $
such that for some $\bar{t}\in I$, $\{J_{1}(\bar{t}),\dots J_{\ell }(\bar{t}%
)\}$ is a basis of $H(t)$. As previously shown, $\{J_{1}(t),\dots J_{\ell
}(t)\}$ are in $H(t)$ for any $t\in I$, and it is a basis of $H(t)$ whenever 
$\Lambda $ has no zeros at that $t$. By continuity, the subspace $\mathcal{H}
$ spanned by $\{J_{1},\dots J_{\ell }\}$ satisfies the third part of the
Lemma.
\end{proof}

\section{Comparison theory for the transverse Jacobi equation\label%
{Jacobi-comp sect}}

\subsection{Riccati Comparison}

In this subsection, we review the Riccati comparison results of Eschenburg (%
\cite{Esch}) and Eschenburg--Heintze (\cite{EschHein}). For $\kappa =-1,0,$
or $1,$ let $\tilde{\lambda}_{\kappa }$ be a solution of the ODE %
\addtocounter{algorithm}{1} 
\begin{equation}
\tilde{\lambda}_{\kappa }^{\prime }+\tilde{\lambda}_{\kappa }^{2}+\kappa =0.
\label{lambda tilde dfn}
\end{equation}%
The possible $\tilde{\lambda}_{\kappa }$ are the logarithmic derivatives of
the functions \addtocounter{algorithm}{1} 
\begin{equation}
\tilde{f}\left( t\right) =\left\{ 
\begin{array}{ll}
\left( c_{1}\sin t+c_{2}\cos t\right) & \text{if }\kappa =1,\vspace{0.05in}
\\ 
\left( c_{1}t+c_{2}\right) & \text{if }\kappa =0,\vspace{0.05in} \\ 
\left( c_{1}\sinh t+c_{2}\cosh t\right) & \text{if }\kappa =-1,%
\end{array}%
\right.  \label{model Jacobi}
\end{equation}%
where $c_{1},c_{2}\in \mathbb{R}$. There are explicit formulas for $\tilde{%
\lambda}_{\kappa }$ in page 302 of \cite{Esch}.

\begin{theorem}
\label{Esch-He thm}(Eschenburg--Heintze, \cite{EschHein}, cf Proposition 2.3
in \cite{Esch}) Let $r:\mathbb{R}\longrightarrow \mathbb{R}$ be a $C^{\infty
}$--function with $r\geq \kappa $. Let $s$ be a smooth solution of the
initial value problem 
\begin{equation*}
s^{\prime }+s^{2}+r=0,\text{ }s\left( t_{0}\right) \leq \tilde{\lambda}%
_{\kappa }\left( t_{0}\right)
\end{equation*}%
on the interval $\left[ t_{0},t_{\max }\right) ,$ where $\tilde{\lambda}%
_{\kappa }$ is as in \eqref{lambda tilde dfn}. Then

\begin{enumerate}
\item \addtocounter{algorithm}{1} 
\begin{equation}
s\left( t\right) \leq \tilde{\lambda}_{\kappa }\left( t\right)
\label{small s
Inequal}
\end{equation}%
on $\left[ t_{0},t_{\max }\right) $.

\item If $s\left( t_{1}\right) =\tilde{\lambda}_{\kappa }\left( t_{1}\right) 
$ for some $t_{1}\in \left( t_{0},t_{\max }\right) ,$ then for all $t\in %
\left[ t_{0},t_{1}\right] $ \addtocounter{algorithm}{1} 
\begin{equation}
s\left( t\right) =\tilde{\lambda}_{\kappa }\left( t\right) \text{ and }r|_{%
\left[ t_{0},t_{1}\right] }\equiv \kappa .  \label{s equals lambda tilde eqn}
\end{equation}
\end{enumerate}
\end{theorem}

When $s$ is the trace of the Riccati operator of a Lagrangian family in $\Ric%
\geq \kappa \left( n-1\right) ,$ the rigidity of Part 2 in Theorem \ref%
{Esch-He thm} also yields rigidity of $S$ and $R\left( \cdot ,\dot{\gamma}%
\right) \dot{\gamma}.$ This idea goes back at least as far as the Splitting
Theorem (\cite{CheegGrom2}) and Cheng's Maximal Diameter Theorem, (\cite%
{Cheng}). It also appears in Croke and Kleiner's paper on rigidity of warped
products (\cite{CrKl}), in Theorem 1.7.1 of \cite{gw}, and in Theorem H of 
\cite{GumWilh}. Since our applications will be to the transverse Jacobi
equation, we formulate them in terms of abstract Riccati equations.

%In order to avoid repeating notation, in Lemmas \ref{Ricci Rigidity Lemma}, \ref{Ricci Rigidity Lemma_t_max_infinity}, and \ref{Ricci Rigidity Lemma_nnc_curvature}, we
%%For $t\in \left[ t_{0},t_{\max }\right) ,$ 
%let $%
%\hat{S}\left( t\right) ,\hat{R}\left( t\right) :V\longrightarrow V$ be
%symmetric endomorphisms of a $k$--dimensional vector space $V$ so that 
%\begin{equation*}
%\hat{S}^{\prime }+\hat{S}^{2}+\hat{R}=0,
%\end{equation*}%
%where $t$ takes values in an interval that we specify in each one of the Lemmas.

\begin{lemma}
\label{Ricci Rigidity Lemma}Let $\hat{S}\left( t\right) ,\hat{R}\left(
t\right) :V\longrightarrow V$ be symmetric endomorphisms of a $k$%
--dimensional vector space $V$ so that on $\left[ t_{0},t_{\max }\right) $ 
\begin{equation*}
\hat{S}^{\prime }+\hat{S}^{2}+\hat{R}=0.
\end{equation*}%
Choose $\tilde{\lambda}_{\kappa }$ a solution of $\tilde{\lambda}_{\kappa
}^{\prime }+\tilde{\lambda}_{\kappa }^{2}+\kappa =0$ defined on $\left[
t_{0},t_{\max }\right) .$ In addition, assume that\addtocounter{algorithm}{1}
\begin{equation}
\Trace\hat{S}\left( t_{0}\right) \leq k\cdot \tilde{\lambda}_{\kappa }\left(
t_{0}\right) ,\text{ and}  \label{initial tr small inequal}
\end{equation}%
\begin{equation*}
\text{ }\Trace\hat{R}\left( t\right) \geq k\cdot \kappa
\end{equation*}%
for all $t\in \left[ t_{0},t_{\max }\right) .$ Then \vspace{0.1in}

\begin{enumerate}
\item For all $t\in \lbrack t_{0},t_{max})$ 
\begin{equation*}
\Trace\hat{S}(t)\leq k\cdot \tilde{\lambda}_{\kappa }(t).
\end{equation*}%
\vspace{0.1in}

\item If\/ $\Trace\hat{S}\left( t_{1}\right) =k\tilde{\lambda}_{\kappa
}\left( t_{1}\right) \/$ for some $t_{1}\in \left( t_{0},t_{\max }\right] ,$
then \addtocounter{algorithm}{1} 
\begin{equation}
\hat{S}\equiv \tilde{\lambda}_{\kappa }\cdot \mathrm{id}\,\text{ and }\,\hat{%
R}=\kappa \cdot \mathrm{id}  \label{rigid S and R eqn}
\end{equation}%
on $\left[ t_{0},t_{1}\right] $, and the solutions of the Jacobi equation $%
J^{\prime \prime }+\hat{R}J=0$ on $\left[ t_{0},t_{1}\right] $, have the form%
\addtocounter{algorithm}{1} 
\begin{equation}
J(t)=\tilde{f}(t)\cdot E,  \label{parellel jacobi}
\end{equation}%
where $E$ is a constant vector in $V$ and $\tilde{f}$ is the function from 
%
%\eqref{f tilde dfn} 
\eqref{model Jacobi} that satisfies $\tilde{f}\left( t_{0}\right)
=\left\vert J\left( t_{0}\right) \right\vert .$
\end{enumerate}
\end{lemma}

\begin{proof}
Set \addtocounter{algorithm}{1} 
\begin{eqnarray}
s &\equiv &\frac{1}{k}\Trace \hat{S} ,  \notag \\
\hat{S}_{0} &\equiv &\hat{S}-\frac{\Trace \hat{S} }{k}\cdot id,\text{ and %
\label{Ricati subs}} \\
r &\equiv &\frac{1}{k}\left( \Trace \hat{R} +\left\vert \hat{S}%
_{0}\right\vert ^{2}\right) .\text{ }  \notag
\end{eqnarray}

Taking the trace of 
\begin{equation*}
\hat{S}^{\prime }+\hat{S}^{2}+\hat{R}=0
\end{equation*}%
yields%
\begin{equation*}
s^{\prime }+s^{2}+r=0.
\end{equation*}%
From inequalities \eqref{small s Inequal} and 
\eqref{initial tr small
inequal}, we get that \addtocounter{algorithm}{1} 
\begin{equation}
s\left( t\right) \leq \tilde{\lambda}_{\kappa }\left( t\right)
\label{small fut eqn}
\end{equation}%
for all $t\in \left( t_{0},t_{\max }\right) $, and the first part follows.

For the second part, if\/ $\Trace\hat{S}\left( t_{1}\right) =k\tilde{\lambda}%
_{\kappa }\left( t_{1}\right)\/ $ for some $t_{1}\in \left( t_{0},t_{\max }%
\right] ,$ then Equation \eqref{s equals lambda tilde eqn} gives us $%
s(t)\equiv \tilde{\lambda}_{\kappa }\left( t\right) $ and $r\equiv \kappa $
in the subinterval $[t_0, t_1]$.

Consequently, 
\begin{equation*}
\kappa =r=\frac{\Trace\hat{R}+|\hat{S}_{0}|^{2}}{k}\geq \frac{\kappa k+|\hat{%
S}_{0}|^{2}}{k}=\kappa +\frac{|\hat{S}_{0}|^{2}}{k}.
\end{equation*}%
%
%
%
%
%
%
%
%
%
%
%
%
%
%
%
%
%
%
%
%
%
%
%\begin{align*}
%\kappa & =r \\
%& =\frac{\Trace \hat{R} +|\hat{S}_{0}|^{2}}{k} \\
%& \geq \frac{\kappa k+|\hat{S}_{0}|^{2}}{k} \\
%& =\kappa +\frac{|\hat{S}_{0}|^{2}}{k}.
%\end{align*}%
Thus $|\hat{S}_{0}|\equiv 0$ and 
\begin{equation*}
\hat{S}=\frac{\Trace\hat{S}}{k}\cdot \id=s\cdot \id=\tilde{\lambda}_{\kappa
}\left( t\right) \cdot \id.
\end{equation*}%
%
%
%
%
%
%
%
%
%
%
%
%
%
%\begin{eqnarray*}
%\hat{S} &=&\frac{\Trace \hat{S} }{k}\cdot id \\
%&=&s\cdot \mathrm{id} \\
%&=&\tilde{\lambda}_{\kappa }\left( t\right) \cdot \mathrm{id}.
%\end{eqnarray*}%
Substituting $\hat{S}=\tilde{\lambda}_{\kappa }\left( t\right) \cdot id$
into the Riccati equation, $\hat{S}^{2}+\hat{S}^{\prime }+\hat{R}=0,$ gives%
\begin{eqnarray*}
(\tilde{\lambda}_{\kappa }^{2}+\tilde{\lambda}_{\kappa }^{\prime })\cdot \id+%
\hat{R} &=&0, \\
-\kappa \cdot \id+\hat{R} &=&0,\text{ and} \\
\hat{R} &=&\kappa \cdot \id.
\end{eqnarray*}%
So the Jacobi fields have the form in equation \eqref{parellel jacobi}.
\end{proof}

\begin{remark}
When $\kappa =0$ and $t_{max}=\infty $, the above Lemma states that if $%
\Trace\hat{S}(t_{0})\leq 0$, then $\Trace\hat{S}(t)\leq 0$ for any $t\geq
t_{0}$, since in this case, $\tilde{\lambda}_{0}\equiv 0$ satisfies
condition \eqref{initial tr small inequal}. The following result improves
this observation.

\begin{lemma}[Long geodesics in nonnegative curvature]
\label{Ricci Rigidity Lemma_nnc_curvature}For $\hat{S}$ and $\hat{R}$ as in
Lemma \ref{Ricci Rigidity Lemma}, suppose that \addtocounter{algorithm}{1} 
\begin{equation}
\Trace\hat{S}\left( t_{0}\right) \leq 0,\text{ and }\Trace\hat{R}\left(
t\right) \geq 0  \label{initial tr small inequal_nnc}
\end{equation}%
for all $t\in \lbrack t_{0},\infty )$. If $\hat{S}$ is defined on $%
[t_{0},\infty ),$ then \addtocounter{algorithm}{1} 
\begin{equation}
\hat{S}\equiv 0\,\text{ and }\,\hat{R}=0,  \label{rigid S and R eqn_nnc}
\end{equation}%
on $\left[ t_{0},\infty \right) .$
\end{lemma}

\begin{proof}
As in the previous proof, \eqref{small s Inequal} gives %
\addtocounter{algorithm}{1} 
\begin{equation}
s\left( t\right) \leq 0  \label{nonpos tr}
\end{equation}%
for all $t\in \left[ t_{0},\infty \right) .$ If for some $t_1>t_0$, $%
s(t_1)<0 $, then there is some $c>t_1$ such that 
\begin{equation*}
s( t_{1}) = \frac{1}{t_1-c} .
\end{equation*}%
Thus, for 
\begin{equation*}
\tilde{\lambda}_{0}(t) =\frac{1}{t-c},
\end{equation*}%
we get from \eqref{small s Inequal} that 
\begin{equation*}
s(t) \leq \tilde{\lambda}_{0}(t)
\end{equation*}%
for all $t\in \left[ t_{1},c\right)$, and in particular, $s(t)$ could not be
defined after $c$. Since this contradicts our hypothesis on $\hat{S}$ being
defined on $\left[ t_{0},\infty \right)$, we obtain that $s\equiv 0$ and $%
r\equiv 0 $ for all $t\in \left( t_{0},\infty\right) .$ The rest of the
proof follows as in Lemma \ref{Ricci Rigidity Lemma}.
\end{proof}
\end{remark}

%\begin{remark}
For arbitrary curvature, there is also a rigidity statement:

\begin{lemma}
\label{Ricci Rigidity Lemma_t_max_infinity} Let $\tilde{\lambda}_{\kappa }$
be as in \eqref{lambda tilde dfn}, and have no singularities on $\left(
t_{0},t_{\max }\right) $. Suppose that \addtocounter{algorithm}{1} 
\begin{equation}
\mathrm{Trace}\,\hat{S}\left( t_{0}\right) \leq k\cdot \tilde{\lambda}%
_{\kappa }\left( t_{0}\right) ,\text{ and }\text{ }\mathrm{Trace}\,\hat{R}%
\left( t\right) \geq k\cdot \kappa  \label{initial tr small inequal_infinity}
\end{equation}%
for all $t\in \left[ t_{0},t_{\max }\right) .$ If $\hat{S}$ is defined on $%
[t_{0},\infty )\ $and 
\begin{equation*}
\lim_{t\rightarrow t_{\max }^{-}}\tilde{\lambda}_{\kappa }\left( t\right)
=-\infty ,
\end{equation*}%
then \addtocounter{algorithm}{1} 
\begin{equation}
\hat{S}\equiv \tilde{\lambda}_{\kappa }\cdot \mathrm{id}\,\text{ and }\,\hat{%
R}=\kappa \cdot \mathrm{id}  \label{rigid S and R eqn_infty}
\end{equation}%
holds on $\left[ t_{0},t_{\max }\right) .$
\end{lemma}

\begin{proof}
The hypothesis $\lim_{t\rightarrow t_{\max }^{-}}\tilde{\lambda}_{\kappa
}\left( t\right) =-\infty $ implies that 
\begin{equation*}
\tilde{\lambda}_{\kappa }\left( t\right) =\left\{ 
\begin{array}{ll}
\cot \left( \pi +t-t_{\max }\right) & \text{if }\kappa =1,\vspace{0.05in} \\ 
\frac{1}{t-t_{\max }} & \text{if }\kappa =0,\vspace{0.05in} \\ 
\coth \left( t-t_{\max }\right) & \text{if }\kappa =-1%
\end{array}%
\right.
\end{equation*}%
(see, e.g., page 302 of \cite{Esch}). Since $\tilde{\lambda}_{\kappa }$ has
no singularities on $\left( t_{0},t_{\max }\right) ,$ it follows that $%
\tilde{\lambda}_{\kappa }\left( t\right) $ is strictly decreasing on $\left(
t_{0},t_{\max }\right) .$ So if $s(t_1)< \tilde{\lambda}_{\kappa }(t_1)$ for
some $t_{1}\in \left( t_{0},t_{\max }\right) $, then there is an $\alpha \in
\left( 0,t_{\max }-t_{1}\right) $ so that 
\begin{equation*}
s\left( t_{1}\right) \leq \tilde{\lambda}_{\kappa }\left( t_{1}+\alpha
\right) .
\end{equation*}%
Thus by \eqref{small s Inequal}, 
\begin{equation*}
s\left( t\right) \leq \tilde{\lambda}_{\kappa }\left( t+\alpha \right)
\end{equation*}%
for all $t\in \left( t_{1},t_{\max }\right) .$ In particular, for some $%
\tilde{t}_{\max }\in \left( t_{0},t_{\max }-\alpha \right] ,$ $%
\lim_{t\rightarrow \tilde{t}_{\max }^{-}}s\left( t\right) =-\infty .$ Since
this contradicts our hypothesis that $\hat{S}$ is defined on $\left(
t_{0},t_{\max }\right) ,$ Inequality \eqref{small fut eqn} must be an
equality for all $t\in \left( t_{0},t_{\max }\right)$ and $r\equiv \kappa .$
\end{proof}

%\end{remark}

\begin{remark}
\label{initiall sing rem}In the event that $\lim_{t\rightarrow t_{0}^{+}}%
\tilde{\lambda}_{\kappa }\left( t\right) =\infty ,$ Theorem \ref{Esch-He thm}
and Lemmas \ref{Ricci Rigidity Lemma}, \ref{Ricci Rigidity
Lemma_nnc_curvature}, and \ref{Ricci Rigidity Lemma_t_max_infinity}, hold
with the hypothesis $s\left( t_{0}\right) =\frac{1}{k}$\textrm{Trace}%
\thinspace\ $\hat{S}\leq \tilde{\lambda}_{\kappa }\left( t_{0}\right) $
replaced by \addtocounter{algorithm}{1} 
\begin{equation}
\lim_{t\rightarrow t_{0}^{+}}\inf \left( \tilde{\lambda}_{\kappa }\left(
t\right) -s\left( t\right) \right) \geq 0.  \label{t_0 =0 inequal}
\end{equation}

If $s$ is the trace of the Riccati operator of the Lagrangian family $%
\left\{ \left. J\text{ }\right\vert \text{ }J\left( t_{0}\right) =0\right\} $
along a geodesic in a Riemannian manifold, then Inequality 
\eqref{t_0 =0
inequal} is satisfied with 
\begin{equation*}
\tilde{\lambda}_{\kappa }\left( t\right) =\left\{ 
\begin{array}{ll}
\cot (t-t_0) & \text{if }\kappa =1 \\ 
\frac{1}{t-t_0} & \text{if }\kappa =0 \\ 
\coth (t-t_0) & \text{if }\kappa =-1%
\end{array}%
\right.
\end{equation*}%
(see Theorem 27 on page 175 of \cite{Pet}). So, for example, in this case,
Theorem \ref{Esch-He thm} implies the classical Rauch Comparison Theorem for 
$2$--manifolds.
\end{remark}

\subsection{Statements of Comparison Lemmas}

For a subspace $\mathcal{W}\subset \Lambda $, write 
\begin{equation*}
\mathcal{W}\left( t\right) =\left\{ \left. J\left( t\right) \text{ }%
\right\vert \text{ }J\in \mathcal{W}\right\} \oplus \left\{ \left. J^{\prime
}\left( t\right) \text{ }\right\vert \text{ }J\in \mathcal{W}\text{ and }%
J\left( t\right) =0\right\} ,
\end{equation*}%
and 
\begin{equation*}
P_{\mathcal{W},t}:\Lambda \left( t\right) \longrightarrow \mathcal{W}\left(
t\right)
\end{equation*}%
for orthogonal projection. For simplicity of notation we will write 
\begin{equation*}
\Trace S_{t}|_{\mathcal{W}}\text{ for }\Trace\left( P_{\mathcal{W},t}\circ
S_{t}|_{\mathcal{W}\left( t\right) }\right) .
\end{equation*}

\begin{remark}
Choose a fixed $t_0\in\mathbb{R}$; given any subspace $W_{t_0}\perp
\gamma^{\prime }(t_0)$, $W_{t_0}$ becomes the horizontal subspace $H(t_0)$
for Wilking's equation when we choose $\mathcal{V}$ as the subset of $%
\Lambda $ formed by Jacobi fields $J$ with $J(t_0)\perp W_{t_0}$.
\end{remark}

By considering $1$--dimensional subspaces, we see that Lemma \ref{sing soon
thm} is a special case of the following result. In its statement we write $%
\Ric_{k}\left( \dot{\gamma},\cdot \right) \geq k\cdot \kappa $ to mean that
the radial intermediate Ricci curvatures along $\gamma $ are bounded from
below by $k\cdot \kappa $, that is, 
\begin{equation*}
\displaystyle\sum\limits_{i=1}^{k}\mathrm{sec}\left( \dot{\gamma}%
,E_{i}\right) \geq k\cdot \kappa
\end{equation*}%
for any orthonormal set $\left\{ \dot{\gamma},E_{1},\ldots ,E_{k}\right\} .$

\begin{lemma}[Intermediate Ricci Comparison]
\label{sing soon Ric_k Lemma} For $\kappa =-1,0,$ or $1,$ let $\gamma
:\left( -\infty ,\infty \right) \longrightarrow M$ be a unit speed geodesic
in a complete Riemannian $n$--manifold $M$ with $\Ric_{k}\left( \dot{\gamma}%
,\cdot \right) $ $\geq k\cdot \kappa .$ Let $\Lambda $ be a Lagrangian
subspace of normal Jacobi fields along $\gamma $ with Riccati operator $S$,
and let $W_{t_{0}}\perp \gamma ^{\prime }(t_{0})$ be some $k$--dimensional
subspace such that %and some Lagrangian subspace $\tilde{\Lambda}$
%along $\tilde{\gamma}$ in $\mathcal{S}_{\kappa }^{2}$, 
\addtocounter{algorithm}{1} 
\begin{equation}
\Trace(S_{t_{0}})|_{W_{t_{0}}}\leq k\cdot \tilde{\lambda}_{\kappa }\left(
t_{0}\right) ,  \label{small initial}
\end{equation}%
where $\tilde{\lambda}_{\kappa }$ is as in (\ref{lambda tilde dfn}). Denote
by $\mathcal{V}$ the subspace of $\Lambda $ formed by those Jacobi fields
that are orthogonal to $W_{t_{0}}$ at $t_{0}$ and by $H(t)$ the subspace of $%
\gamma ^{\prime \perp }$ that is orthogonal to $\mathcal{V}(t)$ at each $%
t\in (t_{0},t_{max})$ 
%(see Remark  \ref{Rem:subspace_induces_Wilking}, part 2).  
. Assume that $\mathcal{V}$ is of full index in the interval $%
[t_{0},t_{max}) $. Then

\begin{enumerate}
\item For all $t\in \left[ t_{0},t_{\max }\right) $, %
\addtocounter{algorithm}{1} 
\begin{equation}
\Trace S_{t}|_{H(t)}\leq k\cdot \tilde{\lambda}_{\kappa }\left( t\right) .
\label{small future Inequal}
\end{equation}

\item If for some $t_{1}\in \lbrack t_{0},t_{max})$, 
\begin{equation*}
\Trace S_{t_{1}}|_{H(t_{1})}=k\cdot \tilde{\lambda}_{\kappa }\left(
t_{1}\right) ,
\end{equation*}%
then the Jacobi equation splits orthogonally along $\gamma $ in the interval 
$[t_{0},t_{1}]$ as 
\begin{equation*}
\Lambda =\mathcal{V}\oplus \mathcal{H}
\end{equation*}%
%
%
%
%
%
%
%
%
%
%
%
%
%
%where $\mathcal{H}(t)=H(t)$ for all $t$, and 
where every nonzero Jacobi field in $\mathcal{H}$ is equal to $J=\tilde{f}%
\cdot E,$ %in $\left[
%t_{0},t_{1}\right] $,
where $E$ is a unit parallel field with $E(t_{0})\in H(t_{0})$, and $\tilde{f%
}$ is the function from \eqref{model Jacobi} %
%{f tilde dfn}
that satisfies $\tilde{f}\left( t_{0}\right) =\left\vert J\left(
t_{0}\right) \right\vert .$
\end{enumerate}
\end{lemma}

\begin{lemma}
\label{lem:riccati comparison minus infinity} Under the hypothesis of the
first part of Lemma \ref{sing soon Ric_k Lemma}, if \/ %For $\kappa
%=-1,0,$ or $1,$ let $\gamma :\left( -\infty ,\infty \right) \longrightarrow
%M $ be a unit speed geodesic in a complete Riemannian $n$--manifold $M$ with 
%$\Ric_{k}$ $\geq k\cdot \kappa .$ Let $\Lambda $ be a Lagrangian subspace of
%normal Jacobi fields along $\gamma $ with Riccati operator $S.$ For $%
%t_{0}<t_{\max },$ suppose that $\Lambda $ has no singularities on 
%$\left(t_{0},t_{\max }\right) $ and that for some $k$--dimensional subspace $%
%W_{t_{0}}\perp \gamma'(t_0) $ ,
%%and some Lagrangian subspace $\tilde{\Lambda}$
%%along $\tilde{\gamma}$ in $\mathcal{S}_{\kappa }^{2}$, 
%\begin{equation}
%\Trace S\left( t_{0}\right) |_{W_{t_0}} \leq k\cdot 
%\tilde{\lambda}_{\kappa }\left( t_{0}\right) ,  \label{small initial}
%\end{equation}%
%where $\tilde{\lambda}_{\kappa }$ is as in (\ref{lambda tilde dfn}).  
$\lim_{t\rightarrow t_{\max }^{-}}\tilde{\lambda}_{\kappa }\left( t\right)
=-\infty $ then the Jacobi equation splits orthogonally along $\gamma $ in
the interval $[t_{0},t_{max})$ as 
\begin{equation*}
\Lambda =\mathcal{V}\oplus \mathcal{H}.
\end{equation*}%
Moreover, every nonzero Jacobi field $J\in \mathcal{H}$ is equal to $J=%
\tilde{f}\cdot E,$ %in $\left[
%t_{0},t_{1}\right] $,
where $E$ is a unit parallel field with $E(t_{0})\in W_{t_{0}}$, and $\tilde{%
f}$ is the function from \eqref{model Jacobi} %
%{f tilde dfn}
that satisfies $\tilde{f}\left( t_{0}\right) =\left\vert J\left(
t_{0}\right) \right\vert .$
\end{lemma}

\begin{lemma}
\label{lem:riccati comparison nnc} Let $\gamma :[t_{0},\infty
)\longrightarrow M$ be a unit speed geodesic in a complete Riemannian $n$%
--manifold $M$ with $\Ric_{k}\left( \dot{\gamma},\cdot \right) \geq 0$. Let $%
\Lambda $ be a Lagrangian subspace of normal Jacobi fields along $\gamma $
with Riccati operator $S.$ Suppose that for some $k$--dimensional subspace $%
W_{t_{0}}\perp \gamma ^{\prime }(t_{0})$ , 
%and some Lagrangian subspace $\tilde{\Lambda}$
%along $\tilde{\gamma}$ in $\mathcal{S}_{\kappa }^{2}$, 
\addtocounter{algorithm}{1} 
\begin{equation}
\Trace S_{t_{0}}|_{W_{t_{0}}}\leq 0.  \label{small initial_nnc}
\end{equation}%
With $\mathcal{V}$ and $H(t)$ as in Lemma \ref{sing soon Ric_k Lemma}, the
Jacobi equation splits orthogonally along $\gamma $ in the interval $%
[t_{0},\infty )$ as 
\begin{equation*}
\Lambda =\mathcal{V}\oplus \mathcal{H}.
\end{equation*}%
Moreover, every nonzero Jacobi field $J\in \mathcal{H}$ is equal to $J=%
\tilde{f}\cdot E,$ %in $\left[
%t_{0},t_{1}\right] $,
where $E$ is a unit parallel field with $E(t_{0})\in W_{t_{0}}$, and $\tilde{%
f}$ is the function from \eqref{model Jacobi} %
%{f tilde dfn}
that satisfies $\tilde{f}\left( t_{0}\right) =\left\vert J\left(
t_{0}\right) \right\vert .$
\end{lemma}

\subsection{Proof of the comparison Lemmas}

In this subsection, we combine Riccati comparison with the Transverse Jacobi
Equation to prove Lemmas \ref{sing soon Ric_k Lemma}, \ref{lem:riccati
comparison minus infinity}, and \ref{lem:riccati comparison nnc}. 
%, and in Subsection \ref{example seubsection}, we
%give an example to show that the fields $J_{0}$ and $J_{1}$ in Lemma \ref%
%{sing soon thm} can indeed be different.

Recall that, for a Lagrangian $\Lambda $ and for fixed $t\in \mathbb{R},$ we
defined the \emph{evaluation map} as%
\begin{eqnarray*}
\mathcal{E}_{t} &:&\Lambda \longrightarrow T_{\gamma \left( t\right) }M \\
\mathcal{E}_{t} &:&J\longmapsto J\left( t\right) .
\end{eqnarray*}

\begin{lemma}
\label{lemma:equality of subspaces} The image of $\mathcal{E}_{t}$ is the
orthogonal complement of the subspace 
\begin{equation*}
\left\{ \,K^{\prime }(t)\,:\,K\in \ker \mathcal{E}_{t}\,\right\} .
\end{equation*}
\end{lemma}

\begin{proof}
Since both subspaces have the same dimension, it suffices to check that for
any $J\in \Lambda $ and any $K\in \ker \mathcal{E}_{t}$, $\langle
J(t),K^{\prime }(t)\rangle =0$; but $\langle J(t),K^{\prime }(t)\rangle
=\langle J^{\prime }(t),K(t)\rangle =0$ since $K(t)=0$.%
%
%If a Jacobi field $J$ vanishes at the point $t$, then 
%\[
%\lim_{s\to t}\frac{J(s)}{s-t}=J'(t).
%\]
\end{proof}

\begin{lemma}
\label{eigem transfer lemma}(Eigenvalue Transfer Lemma) Let $\gamma :\left[
0,l\right] \longrightarrow M$ and $\Lambda $ be as in Lemma \ref{sing soon
Ric_k Lemma}. Let $\mathcal{V}$ be an $\left( n-1-k\right) $--dimensional
subspace of $\Lambda $ with full index in $\left[ 0,l\right]$. %
%for all $t\in \left[ 0,l\right] ,$ the kernel
%of the evaluation map 
%\begin{eqnarray*}
%\mathcal{E}_{t} &:&\Lambda \longrightarrow T_{\gamma \left( t\right) }M \\
%\mathcal{E}_{t} &:&J\longmapsto J\left( t\right)
%\end{eqnarray*}%
%lies in $\mathcal{V}.$ 
For any subspace $\mathcal{W}$ of $\Lambda $, define $\mathcal{W}\left(
t\right) $ as in (\ref{dfn of eval eqn}).

\begin{enumerate}
\item For each fixed $\bar{t}\in \left[ 0,l\right] ,$ there is a $k$%
--dimensional subspace $\mathcal{W}$ of $\Lambda $ so that $\mathcal{W}%
\left( \bar{t}\right) $ is the orthogonal complement of $\mathcal{V}\left( 
\bar{t}\right) .$ If $\mathcal{E}_{\bar{t}}$ is one-to-one, then $\mathcal{W}
$ is unique.

\smallskip

\item Let $\hat{S}_t:H(t)\to H(t)$ be the Riccati operator defined in %
\eqref{dfn of S hat}. Then for any $\mathcal{W}$ as in Part $1,$ 
\begin{equation*}
\Trace\hat{S}_{\bar{t}}=\Trace S_{\bar{t}} |_{\mathcal{W}}.
\end{equation*}
%
%
%\begin{equation*}
% \Trace \hat{S}\vert_{\bar{t}}\leq k\cdot
%\lambda
%\end{equation*}%
%if and only 
%\begin{equation*}
%\left. \Trace\left( S|_{\mathcal{W}}\right) \right\vert _{\bar{t}}\leq k\cdot
%\lambda ,
%\end{equation*}%
where $\Trace S_{\bar{t}} |_{\mathcal{W}}=\mathrm{Trace}\left( P_{\mathcal{W}%
,\bar{t}}\circ S_{\bar{t}}|_{\mathcal{W}\left( \bar{t}\right) }\right) , $
and $P_{\mathcal{W},\bar{t}}:\Lambda \left( \bar{t}\right) \longrightarrow 
\mathcal{W}\left( \bar{t}\right) $ is orthogonal projection.
\end{enumerate}
\end{lemma}

\begin{remark}
For any $\mathcal{W}$ as in Part $1,$ $S_{\bar{t}} |_{\mathcal{W}}$ is well
defined via Remark \ref{S well dfned on W rem}.
\end{remark}

\begin{proof}
Since $\ker \mathcal{E}_{\bar{t}}\subset \mathcal{V}$, we have that 
\begin{equation*}
\left\{ \,J^{\prime }(\bar{t})\,:\,J\in \ker \mathcal{E}_{\bar{t}}\,\right\}
\subset \mathcal{V}(\bar{t}),
\end{equation*}%
and by Lemma \ref{lemma:equality of subspaces}, 
\begin{equation*}
\mathcal{V}(\bar{t})^{\perp }\subset \image\mathcal{E}_{\bar{t}}.
\end{equation*}%
Thus there exist some $k$-dimensional subspace $\mathcal{W}\subset \Lambda $
with $\mathcal{W}(t)=\mathcal{V}(\bar{t})^{\perp }$, and if $\mathcal{E}_{%
\bar{t}}$ is one-to-one, then it is an isomorphism onto $\mathcal{V}\left( 
\bar{t}\right) ^{\perp },$ so $\mathcal{W}$ is unique. 
%Otherwise, let $K\in
%\Lambda $ be in the kernel of $\mathcal{E}_{\bar{t}}$. Then for $J\in
%\Lambda $%
%\begin{equation*}
%\left. \left\langle J,K^{\prime }\right\rangle \right\vert _{\bar{t}}=\left.
%\left\langle J^{\prime },K\right\rangle \right\vert _{\bar{t}}=0.
%\end{equation*}%
%It follows that $\mathcal{E}_{\bar{t}}:\Lambda \longrightarrow T_{\gamma
%\left( \bar{t}\right) }M$ maps onto the orthogonal complement of 
%\begin{equation*}
%\left\{ \left. K^{\prime }\left( \bar{t}\right) \text{ }\right\vert \text{ }%
%K\in \ker \left( \mathcal{E}_{\bar{t}}\right) \right\} .
%\end{equation*}%
%Since $\ker \left( \mathcal{E}_{\bar{t}}\right) \subset \mathcal{V},$ it
%follows that \textrm{im}$\left( \mathcal{E}_{\bar{t}}\right) $ contains the
%orthogonal complement of 
%\begin{equation*}
%\mathcal{V}\left( \bar{t}\right) =\left\{ \left. J\left( \bar{t}\right) 
%\text{ }\right\vert \text{ }J\in \mathcal{V}\right\} \oplus \left\{ \left.
%K^{\prime }\left( \bar{t}\right) \text{ }\right\vert \text{ }K\in \mathcal{V}%
%\text{ and }K\left( \bar{t}\right) =0\right\} ,
%\end{equation*}%
%and Part 1 follows.

To prove Part 2, for $J\in \mathcal{W},$ we write%
\begin{equation*}
J^{\perp }=J-J^{\mathcal{V}},
\end{equation*}%
where $J^{\mathcal{V}}$ is the component of $J$ that lies in $\mathcal{V}%
\left( t\right) .$ Then for all $t,$%
\begin{equation*}
0=\frac{d}{dt}\left\langle J^{\mathcal{V}},J^{\perp }\right\rangle
=\left\langle \left( J^{\mathcal{V}}\right) ^{\prime },J^{\perp
}\right\rangle +\left\langle J^{\mathcal{V}},J^{\perp \prime }\right\rangle .
\end{equation*}

Since $J\in \mathcal{W},$ $J^{\mathcal{V}}\left( \bar{t}\right) =0,$ and $%
\left\langle J^{\mathcal{V}},J^{\perp \prime }\right\rangle |_{\bar{t}}=0.$
So the previous display evaluated at $\bar{t}$ becomes%
\begin{equation*}
\left. \left\langle \left( J^{\mathcal{V}}\right) ^{\prime },J^{\perp
}\right\rangle \right\vert _{\bar{t}}=0.
\end{equation*}%
For $J\in \mathcal{W},$ it follows that \addtocounter{algorithm}{1} 
\begin{eqnarray}
\left. \left\langle \hat{S}\left( J^{\perp }\right) ,J^{\perp }\right\rangle
\right\vert _{\bar{t}} &=&\left. \left\langle \left( J^{\prime }-\left( J^{%
\mathcal{V}}\right) ^{\prime }\right) ,J^{\perp }\right\rangle \right\vert _{%
\bar{t}}  \notag \\
&=&\left. \left\langle J^{\prime },J^{\perp }\right\rangle \right\vert _{%
\bar{t}}  \notag \\
&=&\left. \left\langle S\left( J\right) ,J^{\perp }\right\rangle \right\vert
_{\bar{t}}  \notag \\
&=&\left. \left\langle S\left( J\right) ,J\right\rangle \right\vert _{\bar{t}%
}.  \label{S hat vs S}
\end{eqnarray}%
So 
\begin{equation*}
\Trace\hat{S}_{\bar{t}}=\Trace S_{\bar{t}}|_{\mathcal{W}(\bar{t})}.
\end{equation*}
\end{proof}

\begin{proof}[Proof of Lemma \protect\ref{sing soon Ric_k Lemma}]
We combine Theorem \ref{Esch-He thm} and Lemma \ref{Ricci Rigidity Lemma}
with the Transverse Jacobi Equation, and the Eigenvalue Transfer Lemma \ref%
{eigem transfer lemma}.

Observe first that Theorem \ref{Esch-He thm} holds on intervals where the
function $s$ is smooth. In our context, this happens as long as $\hat{S}$ is
well-defined. According to Proposition \ref{S hat well defined prop}, $\hat{S%
}$ is well-defined at all times $t$ where $\mathcal{V}$ has full index at $t$%
, and therefore we can apply it in the situation of Lemma \ref{sing soon
Ric_k Lemma}.

Recall that 
\begin{equation*}
\mathcal{V}\equiv \left\{ \left. X\in \Lambda \text{ }\right\vert \text{ }%
X\left( t_{0}\right) \perp J\left( t_{0}\right) \text{ for all }J\in
W_{t_{0}}\right\} .
\end{equation*}

Let $\hat{S}:H\left( t\right) \longrightarrow H\left( t\right) $ be as in
Equation \eqref{dfn of S hat}. It follows from the Eigenvalue Transfer Lemma %
\ref{eigem transfer lemma} that 
\begin{equation*}
\Trace\hat{S}_{t_{0}}\leq k\cdot \tilde{\lambda}_{\kappa }\left(
t_{0}\right) .
\end{equation*}%
The Transverse Jacobi Equation says,\addtocounter{algorithm}{1} 
\begin{equation}
\hat{S}^{\prime }+\hat{S}^{2}+\left\{ R\left( \cdot ,\dot{\gamma}(t)\right) 
\dot{\gamma}(t)\right\} ^{h}+3AA^{\ast }=0.  \label{Wilk-Gr-On eqn}
\end{equation}

Since $\Ric_{k}$ $\geq k$, $AA^{\ast }$ is nonnegative, and $W_{t_{0}}$ is $%
k $--dimensional, when we take the trace of Equation \eqref{Wilk-Gr-On eqn},
divide by $k,$ and make the substitutions of \eqref{Ricati
subs}, we get an equation that satisfies the hypotheses of Theorem \ref%
{Esch-He thm}. Thus for all $t\in \left[ t_{0},t_{\max }\right) ,$ %
\addtocounter{algorithm}{1} 
\begin{equation}
\frac{1}{k}\Trace\hat{S}_t \leq \tilde{\lambda}_{\kappa }\left( t\right) .
\label{Tr S hat eqn}
\end{equation}

By combining this with the Eigenvalue Transfer Lemma \ref{eigem transfer
lemma} and the fact that $\mathcal{V}$ has full index on $\left(
t_{0},t_{\max }\right) ,$ we have 
\begin{equation*}
\Trace S|_{H(t)}\leq k\cdot \lambda _{\kappa },
\end{equation*}%
as claimed.

To prove the rigidity statement, suppose that 
\begin{equation*}
\Trace S|_{H(t_1)} =k\cdot \tilde{\lambda}_{\kappa }\left( t_{1}\right)
\end{equation*}%
for some $t_{1}\in \left( t_{0},t_{\max }\right) .$

It follows from Lemma \ref{eigem transfer lemma} that 
\begin{equation*}
\Trace\hat{S}_{t_1}=k\cdot \tilde{\lambda}_{\kappa }\left( t_{1}\right) .
\end{equation*}

Writing $\hat{R}$ for \/$\left\{ R\left( \cdot ,\dot{\gamma}(t)\right) \dot{%
\gamma}(t)\right\} ^{h}+3AA^{\ast }$, %\begin{equation*}
%\hat{R}\text{ for }\left\{ R\left( \cdot ,\dot{\gamma}(t)\right) \dot{\gamma}%
%(t)\right\} ^{h}+3AA^{\ast },
%\end{equation*}%
we see from Theorem \ref{Esch-He thm} that 
\begin{equation*}
\Trace\hat{S}_t\equiv k\cdot \tilde{\lambda}_{\kappa }\left( t\right) \text{
and }\Trace \hat{R} \equiv k\cdot \kappa
\end{equation*}%
for all $t\in \left[ t_{0},t_{1}\right] .$

Our hypothesis that $\Ric_{k}$ $\geq k\cdot \kappa $ implies that $\Trace%
\left\{ R\left( \cdot ,\dot{\gamma}(t)\right) \dot{\gamma}(t)\right\}
^{h}\geq k\cdot \kappa .$ Combining this with $\Trace\hat{R}\equiv k\cdot
\kappa $ and the fact that $AA^{\ast }$ is nonnegative, we see that $A\equiv
0$. So Lemma \ref{lem:Lagrangian splitting} guarantees the existence of a
subspace $\mathcal{H}$ in $\Lambda $ such that $\mathcal{H}(t)=H(t)$ at
every $t\in \lbrack t_{0},t_{1}]$, and $\Lambda $ splits orthogonally as 
\begin{equation*}
\Lambda =\mathcal{V}\oplus \mathcal{H}.
\end{equation*}

By Part 2 of Lemma $\ref{Ricci Rigidity Lemma},$ $\hat{S}\equiv \tilde{%
\lambda}_{\kappa }\cdot \mathrm{id}$ and $\hat{R}=\kappa \cdot \mathrm{id.}$
So it follows that $\mathcal{H}$ consists of Jacobi fields whose
restrictions to $\left[ t_{0},t_{1}\right] $ have the form%
\begin{equation*}
J=\tilde{f}E,
\end{equation*}%
where $E$ is a parallel field and $\tilde{f}$ is the function from \eqref
%{f tilde dfn}
{model Jacobi} that satisfies $\tilde{f}\left( t_{0}\right) =\left\vert
J\left( t_{0}\right) \right\vert .$
\end{proof}

\begin{proof}[Proof of Lemma \protect\ref{lem:riccati comparison minus
infinity}]
Since $\mathcal{V}$ has full index, Proposition \ref{S hat well defined prop}
implies that $\hat{S}$ is defined on $\left[ t_{0},t_{\max }\right) $. As
above, the Eigenvalue Transfer Lemma \ref{eigem transfer lemma} gives us
that 
\begin{equation*}
\Trace\hat{S}_{t_{0}}\leq k\cdot \tilde{\lambda}_{\kappa }\left(
t_{0}\right) .
\end{equation*}

%
%
%To prove Part 3, we suppose that $\Lambda $ has no singularities on $\left[
%t_{0},t_{\max }\right) ,$ so by Proposition \ref{S hat well defined prop}, $%
%\hat{S}$ is defined on $\left[ t_{0},t_{\max }\right) .$ As above, the Eigen
%Transfer Lemma \ref{eigem transfer lemma} gives us that 
%\begin{equation*}
%\left. \mathrm{Trace}\left( \hat{S}\right) \right\vert _{t_{0}}\leq k\cdot 
%\tilde{\lambda}_{\kappa }\left( t_{0}\right) .
%\end{equation*}
Once again, $\Ric_{k}\geq k\cdot \kappa $ implies that $\Trace\left\{
R\left( \cdot ,\dot{\gamma}(t)\right) \dot{\gamma}(t)\right\} ^{h}\geq
k\cdot \kappa $ and $\hat{R}\geq k\cdot \kappa $. So by Lemma \ref{Ricci
Rigidity Lemma_t_max_infinity}, $\hat{S}\equiv \tilde{\lambda}_{\kappa
}\cdot \mathrm{id}$ and $\hat{R}=\kappa \cdot \mathrm{id}$ on $\left[
t_{0},t_{\max }\right) $. This implies, as in the proof of Part 2 of Lemma %
\ref{sing soon Ric_k Lemma}, that $A=0$ in $[t_{0},t_{max})$. The remainder
of the argument is exactly the same as the proof of Part 2 of Lemma \ref%
{sing soon Ric_k Lemma}.%As before, our hypothesis that $
%\Ric_{k}$ $\geq k\cdot \kappa $ implies that $\mathrm{Trace}\left\{ R\left(
%\cdot ,\dot{\gamma}(t)\right) \dot{\gamma}(t)\right\} ^{h}\geq k\cdot \kappa
%.$ Combining this with $\hat{R}\equiv k\cdot \kappa $ and the fact that $%
%AA^{\ast }$ is nonnegative, we see that $A\equiv 0$. So%
%\begin{equation*}
%W_{t}\equiv \left\{ \left. J\in \Lambda \text{ }\right\vert \text{ }J\left(
%t\right) \perp \mathcal{V}\left( t\right) \right\}
%\end{equation*}%
%is independent of $t$, and $W_{t}=$ $W_{t_{0}}$ for all $t.$ It follows that 
%$\Lambda $ splits orthogonally as%
%\begin{equation*}
%\Lambda \equiv W_{t_{0}}\oplus \mathcal{V}.
%\end{equation*}
%
%Since $\hat{S}\equiv \tilde{\lambda}_{\kappa }\cdot \mathrm{id}$ and $\hat{R}%
%=\kappa \cdot \mathrm{id,}$ as in Part 2, we have that $W_{t_{0}}$ consists
%of Jacobi fields whose restrictions to $\left[ t_{0},t_{1}\right] $ have the
%form%
%\begin{equation*}
%J=\tilde{f}E,
%\end{equation*}%
%where $E$ is a parallel field and $\tilde{f}$ is the function from %
%\eqref{model Jacobi} %{f tilde dfn}
%that satisfies $\tilde{f}\left( t_{0}\right) =\left\vert J\left(
%t_{0}\right) \right\vert .$
\end{proof}

\begin{proof}[Proof of Lemma \protect\ref{lem:riccati comparison nnc}]
Since $\mathcal{V}$ has full index, Proposition \ref{S hat well defined prop}
gives that $\hat{S}$ is defined on $\left[ t_{0},\infty \right) .$ As above,
the Eigenvalue Transfer Lemma \ref{eigem transfer lemma} gives us that 
\begin{equation*}
\Trace\hat{S}_{t_{0}}\leq 0.
\end{equation*}

So by Lemma \ref{Ricci Rigidity Lemma_nnc_curvature}, $\hat{S}\equiv 0$ and $%
\hat{R}\equiv 0$ on $\left[ t_{0},\infty \right) .$ As before, our
hypothesis that $\Ric_{k}$ $\geq 0$ implies that $\mathrm{Trace}\left\{
R\left( \cdot ,\dot{\gamma}(t)\right) \dot{\gamma}(t)\right\} ^{h}\geq 0.$
Combining this with $\hat{R}\equiv 0$ and the fact that $AA^{\ast }$ is
nonnegative, we see that $A\equiv 0.$ The remainder of the argument is
exactly the same as the proof of Part 2 of Lemma \ref{sing soon Ric_k Lemma}.

%So by Part 2 of Lemma $\ref{Ricci Rigidity Lemma},$ $\hat{S}\equiv \tilde{%
%\lambda}_{\kappa }\cdot \mathrm{id}\equiv 0$ and $\hat{R}=\kappa \cdot 
%\mathrm{id}\equiv 0$ on $\left[ t_{0},\infty \right) .$ As before, our
%hypothesis that $\Ric_{k}$ $\geq 0$ implies that $\mathrm{Trace}%
%\left\{ R\left( \cdot ,\dot{\gamma}(t)\right) \dot{\gamma}(t)\right\}
%^{h}\geq 0.$ Combining this with $\hat{R}\equiv 0$ and the fact that $%
%AA^{\ast }$ is nonnegative, we see that $A\equiv 0$. So%
%\begin{equation*}
%W_{t}\equiv \left\{ \left. J\in \Lambda \text{ }\right\vert \text{ }J\left(
%t\right) \perp \mathcal{V}\left( t\right) \right\}
%\end{equation*}%
%is independent of $t$, and $W_{t}=$ $W_{t_{0}}$ for all $t.$ It follows that 
%$\Lambda $ splits orthogonally as%
%\begin{equation*}
%\Lambda \equiv W_{t_{0}}\oplus \mathcal{V}.
%\end{equation*}
%
%Since $\hat{S}\equiv 0$ and $\hat{R}=0\mathrm{,}$ as in Part 2, we have that 
%$W_{t_{0}}$ consists of Jacobi fields whose restrictions to $\left[
%t_{0},t_{1}\right] $ have the form%
%\begin{equation*}
%J=E,
%\end{equation*}%
%where $E$ is a parallel field.
\end{proof}

%\begin{corollary}
%\label{W(s) perp to V(s) cor}Let $\mathcal{V}$ be as in the proof of Lemma %
%\ref{sing soon Ric_k Lemma}. In Part 1 of Lemma \ref{sing soon Ric_k Lemma},
%we can choose the subspace $W_{t}\subset \Lambda $ so that 
%\begin{equation*}
%W_{t}\left( t\right) \perp \mathcal{V}\left( t\right) .
%\end{equation*}
%\end{corollary}

\begin{remark}
\label{iniital sing remark}If $\lim_{t\rightarrow t_{0}^{+}}\tilde{\lambda}%
_{\kappa }\left( t\right) =\infty ,$ then, using Remark \ref{initiall sing
rem}, Lemmas \ref{sing soon Ric_k Lemma} to \ref{lem:riccati comparison nnc}
hold with the hypothesis $\Trace(S_{t_0})|_{W_{t_0}} \leq k\tilde{\lambda}%
_{\kappa }\left( t_0\right) $ replaced with \addtocounter{algorithm}{1} 
\begin{equation}
\lim_{t\rightarrow t_{0}^{+}}\inf \left( \Trace S_t|_{H(t)} -k\tilde{\lambda}%
_{\kappa }\left( t\right) \right) \geq 0,  \label{sing 0 inequal}
\end{equation}%
and 
\begin{equation*}
\tilde{\lambda}_{\kappa }\left( t\right) =\left\{ 
\begin{array}{ll}
\cot (t-t_0) & \text{if }\kappa =1 \\ 
\frac{1}{t-t_0} & \text{if }\kappa =0 \\ 
\coth (t-t_0) & \text{if }\kappa =-1.%
\end{array}%
\right.
\end{equation*}

If $N$ is a smooth submanifold of $M$, then Inequality \eqref{sing 0 inequal}
holds for 
\begin{equation*}
W_{0}=\left\{ J|J\left( 0\right) =0\text{, }J^{\prime }\left( 0\right) \in
\nu _{\gamma \left( 0\right) }\left( N\right) \right\} \subset \Lambda _{N}
\end{equation*}%
(see Part 3 of Lemma 2.7 in \cite{SearW} and also Remark 3 in \cite{EschHein}%
).
\end{remark}

\subsection{Why $J_{1}$ need not be $J_{0}$}

\label{example seubsection} This subsection neither depends on nor is used
in the rest of the paper. In it we give examples showing that the field $%
J_{1}$ in Lemma \ref{sing soon thm} can indeed be different from the field $%
J_{0}.$ A similar example can be found on page 463 of \cite{HeiKar}.

\begin{example}
\label{diff future ex}Let $E_{1}$ and $E_{2}$ be parallel orthonormal fields
along a geodesic $\gamma $ in $\mathbb{R}^{3}$ with $E_{1},E_{2}\perp \gamma
.$ Let $\Lambda $ be the Lagrangian family 
\begin{equation*}
\Lambda =\mathrm{span}\left\{ tE_{1},\text{ }\left( t+1\right) E_{2}\right\}
.
\end{equation*}%
Let 
\begin{equation*}
J_{0}=tE_{1}+\left( t+1\right) E_{2}.
\end{equation*}%
Then%
\begin{equation*}
\left\langle J_{0}^{\prime }\left( 0\right) ,J_{0}\left( 0\right)
\right\rangle =\tilde{\lambda}_{0}\left( 0\right) \left\langle J_{0}\left(
0\right) ,J_{0}\left( 0\right) \right\rangle =1,
\end{equation*}%
where $\tilde{\lambda}_{0}=\frac{1}{t+1}$ comes from the model Jacobi field
on $\mathbb{R}^{2}$ given by $\tilde{J}=\left( t+1\right) \tilde{E}$ with $%
\tilde{E}$ a parallel field$.$ In particular, $J_{0}$ satisfies Inequality (%
\ref{sec iniitial cond}) with $t_{0}=0.$

On the other hand, 
\begin{equation*}
\left\langle J_{0}^{\prime }\left( t\right) ,J_{0}\left( t\right)
\right\rangle =\left\langle E_{1}+E_{2},tE_{1}+\left( t+1\right)
E_{2}\right\rangle =2t+1,
\end{equation*}%
and for $t>0,$ 
\begin{equation*}
\tilde{\lambda}_{0}\left( t\right) \left\langle J_{0}\left( t\right)
,J_{0}\left( t\right) \right\rangle =\frac{1}{t+1}\left( t^{2}+\left(
t+1\right) ^{2}\right) =\frac{t^{2}}{t+1}+t+1<2t+1=\left\langle
J_{0}^{\prime }\left( t\right) ,J_{0}\left( t\right) \right\rangle .
\end{equation*}%
To verify the validity of Lemma \ref{sing soon thm} for this example, take $%
J_{1}\left( t\right) =\left( t+1\right) E_{2}$ and note that Inequality \ref%
{sm fut sec Ineq} is an equality for all $t>0.$
\end{example}

\begin{example}
\label{hopf holn}Let $E_{1}$ and $E_{2}$ be parallel orthonormal fields
along a geodesic $\gamma $ in $\mathbb{S}^{3}$ with $E_{1},E_{2}\perp \gamma
.$ Let $\Lambda $ be the Lagrangian family 
\begin{equation*}
\Lambda =\mathrm{span}\left\{ \sin tE_{1},\text{ }\cos tE_{2}\right\} .
\end{equation*}%
Let 
\begin{equation*}
J=\sin tE_{1}+\cos tE_{2}.
\end{equation*}%
Then%
\begin{equation*}
\left\langle J^{\prime }\left( 0\right) ,J\left( 0\right) \right\rangle =0.
\end{equation*}%
So $J$ satisfies Inequality \ref{sec iniitial cond} where $\tilde{\lambda}%
=\cot \left( t+\frac{\pi }{2}\right) $ comes from the model Jacobi field on $%
\mathbb{S}^{2}$ given by $\tilde{J}=\cos \left( t\right) \tilde{E}$ with $%
\tilde{E}$ a parallel field$.$ On the other hand, for $t\in \left( 0,\frac{%
\pi }{2}\right) ,$ 
\begin{eqnarray*}
\left\langle J^{\prime }\left( t\right) ,J\left( t\right) \right\rangle
&\equiv &0 \\
&>&\cot \left( t+\frac{\pi }{2}\right) \left\langle J\left( t\right)
,J\left( t\right) \right\rangle ,
\end{eqnarray*}%
and Inequality \ref{sm fut sec Ineq} does not hold with $J_{0}=J_{1}=J$, $%
\tilde{\lambda}=\cot \left( t+\frac{\pi }{2}\right) ,$ and $t\in \left( 0,%
\frac{\pi }{2}\right) .$

In contrast, the field $\cos \left( t+\frac{\pi }{2}\right) E_{2}$ satisfies
Inequality \ref{sm fut sec Ineq} for all $t\in \left( 0,\frac{\pi }{2}%
\right) .$
\end{example}

\section{Focal Radius and Positive Curvature\label{focal and pos sect}}

In this section, we prove Theorem \ref{Intermeadiate Ricci Thm}, and give
examples showing the hypotheses on the dimension of $N$ can not be removed.

\begin{proof}[Proof of Theorem \protect\ref{Intermeadiate Ricci Thm} (cf
Theorem 3.5 in \protect\cite{GonGui})]
Let $v\in \nu \left( N\right) $ be any unit vector. Recall that we denoted
by $\Lambda _{N}$ the Lagrangian of normal Jacobi fields along $\gamma _{v}$
given by 
\begin{equation*}
\Lambda _{N}=\left\{ J\,|\,J\left( 0\right) \in T_{\gamma _{v}\left(
0\right) }N\text{ and }J^{\prime }\left( 0\right) =\mathrm{S}_{v}J\left(
0\right) \right\} .
\end{equation*}%
It suffices to show that for the subspace 
\begin{eqnarray*}
\mathcal{K} &\equiv &\spann\left\{ \left. J\in \Lambda _{N}\text{ }%
\right\vert \text{ }J\left( t_{i}\right) =0\text{ for some nonzero }t_{i}\in %
\left[ -\frac{\pi }{2},\frac{\pi }{2}\right] \right\} , \\
\dim \mathcal{K} &\geq &\mathrm{\dim }\left( N\right) -k+1.
\end{eqnarray*}%
The definition of $\mathcal{K}$ implies that $\mathcal{K}$ has full index
for all $t\in \left[ -\frac{\pi }{2},\frac{\pi }{2}\right] $.

Suppose, by way of contradiction, that $\dim \mathcal{K}\leq \mathrm{\dim }%
\left( N\right) -k,$ and set 
\begin{equation*}
\mathcal{K}\left( t\right) \equiv \left\{ \left. J\left( t\right) \text{ }%
\right\vert \text{ }J\in \mathcal{K}\right\} \oplus \left\{ \left. J^{\prime
}\left( t\right) \text{ }\right\vert \text{ }J\in \mathcal{K}\text{ and }%
J\left( t\right) =0\right\} .
\end{equation*}

Since $\dim \mathcal{K}\leq \mathrm{\dim }\left( N\right) -k,$ there is a $k$%
--dimensional subspace $W_{0}\subset T_{\gamma _{v}\left( 0\right) }N$
orthogonal to $\mathcal{K}\left( 0\right) $. Replacing $\gamma _{v}$ with $%
\gamma _{-v}$ if necessary we may assume that \addtocounter{algorithm}{1} 
\begin{equation}
\Trace\left( S_{0}|_{W_{0}}\right) \leq 0.  \label{trace 0 ineq}
\end{equation}

Let $\mathcal{V}\subset \Lambda _{N}$ be the subspace so that $\mathcal{V}%
\left( 0\right) \perp W\left( 0\right) ,$ and notice that $\mathcal{K}%
\subset \mathcal{V}$. From \eqref{trace 0 ineq}, we see that Lemma \ref{sing
soon Ric_k Lemma} applies to $\Lambda _{N}$ and $W_{0}$ on $\left[ 0,\frac{%
\pi }{2}\right] $. So for all $t\in \left[ 0,\frac{\pi }{2}\right] ,$ there
is a $k$--dimensional space $H(t)\subset \gamma _{v}^{\prime }(t)^{\perp }$
so that \addtocounter{algorithm}{1} 
\begin{equation}
\Trace S_{t}|_{H(t)}\leq k\cot \left( t+\frac{\pi }{2}\right) \text{ and }%
H(t)\perp \mathcal{V}\left( t\right) .
\label{small future prf lem Q Inequal}
\end{equation}%
It follows from Inequality \eqref{small
future prf lem Q Inequal} that there is a $Z\in \Lambda \setminus \mathcal{V}
$ with\addtocounter{algorithm}{1} 
\begin{equation}
Z\left( t\right) =0\text{ for some }t\in \left( 0,\frac{\pi }{2}\right] .
\label{Z_i sing}
\end{equation}%
Since $Z\notin \mathcal{V},$ it follows that $Z\notin \mathcal{K},$ and %
\eqref{Z_i sing} contradicts the definition of $\mathcal{K}.$

To prove Part 2, assume that the focal radius of $N$ is $\frac{\pi }{2}.$ If
necessary we replace $\gamma _{v}$ with $\gamma _{-v}$ to arrange that 
\begin{equation*}
\Trace\left( S_{0}|_{W_{0}}\right) \leq 0.
\end{equation*}%
This allows us to apply Lemma \ref{lem:riccati comparison minus infinity}
with $W_{0}=T_{\gamma _{v}\left( 0\right) }N$, $\kappa =1,$ $t_{0}=0,$ $%
t_{\max }=\frac{\pi }{2},$ and $\tilde{\lambda}_{1}=\cot \left( t+\frac{\pi 
}{2}\right) $, to conclude that 
\begin{equation*}
\left\{ J\,|\,J\left( 0\right) \in T_{\gamma _{v}\left( 0\right) }N\text{
and }J^{\prime }\left( 0\right) =\mathrm{S}_{v}J\left( 0\right) \right\}
\end{equation*}%
is spanned by Jacobi fields of the form $\sin \left( t+\frac{\pi }{2}\right)
E$ where $E$ is a parallel field. In particular, $\mathrm{S}_{v}\equiv 0,$
and since this holds for all unit vectors $v$ orthogonal to $N,$ $N$ is
totally geodesic.
\end{proof}

\begin{remark}
\label{not Berger}Although the Ricci curvature version of Theorem \ref%
{Intermeadiate Ricci Thm} can be proven via standard Riccati comparison (see
e.g. \cite{EschHein}), its statement does not seem to be in the literature.
In contrast, it does not seem possible to prove the sectional curvature
version of Theorem \ref{Intermeadiate Ricci Thm} with existing Jacobi or
Riccati comparison results. In the special case when $N$ is known to be
totally geodesic, there are $J$ in $\Lambda _{N}$ with $J^{\prime }\left(
0\right) =0.$ Berger's version of the Rauch Comparison Theorem then gives $%
\left\langle J^{\prime }\left( t\right) ,J\left( t\right) \right\rangle \leq
\cot \left( \frac{\pi }{2}+t\right) $ for all $t\in \left( 0,\frac{\pi }{2}%
\right) .$ In particular, $\gamma $ would have a focal point in $\left[ 0,%
\frac{\pi }{2}\right] $ (see Theorem 1.29 in \cite{CheegEbin} and Theorem
4.9 on page 234 of \cite{doCarm})$.$

For a general submanifold, we can always flip the parameterization of a
geodesic as in the proof of Theorem \ref{Intermeadiate Ricci Thm}, to obtain 
$\left\langle J^{\prime }\left( 0\right) ,J\left( 0\right) \right\rangle
\leq 0$ for some $J\in \Lambda _{N}.$ However, Example \ref{hopf holn} shows
that $\left\langle J^{\prime }\left( t\right) ,J\left( t\right)
\right\rangle $ can exceed $\cot \left( \frac{\pi }{2}+t\right) $ if $%
J^{\prime }\left( 0\right) \neq 0.$ In fact, the $J$ of Example \ref{hopf
holn} never vanishes! Thus it does not seem possible to prove Theorem \ref%
{Intermeadiate Ricci Thm} using only Berger's Theorem in place of Lemma \ref%
{sing soon Ric_k Lemma}.
\end{remark}

\subsection{Examples}

\label{optimal examples for theorems A and B} Next we give examples showing
that the hypotheses about the dimension of the submanifolds in Theorems \ref%
{Intermeadiate Ricci Thm} and \ref{inter Ricci Rigidity thm} cannot be
removed. For the sectional curvature versions of the theorems, a point in
small perturbation of $\mathbb{S}^{n}$ shows that the conclusions can be
false if $N$ does not have positive dimension. For the Ricci curvature
versions of the theorems, we have the following examples.

\begin{example}
Let $S_{k}^{n}$ be the $n$--sphere with constant curvature $k.$ The product
metric on $S_{\frac{n+1}{n-1}}^{n}\times S_{n+1}^{2}$ satisfies %
\addtocounter{algorithm}{1} 
\begin{eqnarray}
\Ric\left( S_{\frac{n+1}{n-1}}^{n}\times S_{n+1}^{2}\right) &=&n+1=\Ric%
\left( \mathbb{S}^{n+2}\right) ,\text{ and\label{gen ricci ex}} \\
\mathrm{FocalRadius}\left( \left\{ pt\right\} \times S_{n+1}^{2}\right)
&=&\pi \sqrt{\frac{n-1}{n+1}}\longrightarrow \pi \text{ as }n\rightarrow
\infty .  \notag
\end{eqnarray}

Thus the focal radius of $N$ in the Ricci curvature version of Theorem \ref%
{Intermeadiate Ricci Thm} can converge to $\pi $ if the hypothesis that $N$
is a hypersurface is removed and the dimension of $M$ is allowed to go to $%
\infty ,$ while the dimension of $N$ is fixed.

On the other hand, if we take $n=2$ or $3$, then \eqref{gen ricci ex}
becomes 
\begin{eqnarray*}
\Ric\left( S_{3}^{2}\times S_{3}^{2}\right) &\equiv &3\equiv \Ric\left( 
\mathbb{S}^{4}\right) ,\text{ } \\
\mathrm{FocalRadius}\left( \left\{ pt\right\} \times S_{3}^{2}\right) &=&\pi 
\sqrt{\frac{1}{3}}>\frac{\pi }{2},\text{ and}
\end{eqnarray*}%
\begin{eqnarray*}
\Ric\left( S_{2}^{3}\times S_{4}^{2}\right) &=&4=\Ric\left( \mathbb{S}%
^{5}\right) \text{,} \\
\mathrm{FocalRadius}\left( S_{2}^{3}\times \left\{ pt\right\} \right) &=&%
\frac{\pi }{2}.
\end{eqnarray*}%
So the hypothesis that $N$ is a hypersurface in Ricci curvature versions of
Theorem \ref{Intermeadiate Ricci Thm} cannot be replaced with the hypothesis
that $N$ is a codimension $2$ submanifold. Similarly, in the Ricci curvature
version of Theorem \ref{inter Ricci Rigidity thm}, the hypersurface can not
be replaced with a codimension 2 submanifold.
\end{example}

For our intermediate Ricci curvature results we have

\begin{example}
For $k>\frac{4}{3}p$ and $p\geq 2,$ $M=S_{\frac{k}{k-p}}^{k-1}\times
S_{k}^{p}$ satisfies 
\begin{eqnarray*}
\Ric_{k}\left( M\right) &\geq &k\text{ and} \\
\mathrm{FocalRadius}\left( \left\{ pt\right\} \times S_{k}^{p}\right) &=&\pi 
\sqrt{\frac{k-p}{k}}>\frac{\pi }{2}\text{, if }k>\frac{4}{3}p.
\end{eqnarray*}

Thus $\left\{ pt\right\} \times S_{k}^{p}\subset S_{\frac{k}{k-p}%
}^{k-1}\times S_{k}^{p}$ is a closed submanifold of a $\left( k+p-1\right) $%
--manifold with $\Ric_{k}\left( M\right) \geq k$ and focal radius $>\frac{%
\pi }{2},$ and the focal radius of $N$ in Theorem \ref{Intermeadiate Ricci
Thm} can exceed $\frac{\pi }{2}$ if the hypothesis that $\dim \left(
N\right) \geq k$ is replaced with $\dim \left( N\right) \geq p$ where $\frac{%
3}{4}k>p. $

By sending $k\rightarrow \infty $ while keeping $p$ fixed, we see that 
\begin{equation*}
\mathrm{FocalRadius}\left( \left\{ pt\right\} \times S_{k}^{p}\right) =\pi 
\sqrt{\frac{k-p}{k}}\longrightarrow \pi .
\end{equation*}%
So in Theorem \ref{Intermeadiate Ricci Thm}, the focal radius of $N$ can
converge to $\pi ,$ if there is no hypothesis about the dimension of $N,$
and the dimension of $M$ is allowed to go to $\infty .$
\end{example}

\part*{{\protect\Large Part 2: Focal Rigidity\label{focal rig Part}}}

%\part{{\large Focal Rigidity\label{focal rig Part}}}

Let $M$ be a complete Riemannian manifold $M$ with $\Ric_{k}\geq k,$ and let 
$N$ be a closed submanifold of $M$ of dimension at least $k$ and focal
radius $\frac{\pi }{2}.$ Since each connected component of $N$ has focal
radius $\frac{\pi }{2},$ we may assume that $N$ is connected.

In the second part of the paper, we prove Theorem \ref{inter Ricci Rigidity
thm} by showing that the universal cover of $M$ is isometric to the unit
sphere or to a projective space with the standard metric, with $N$ totally
geodesic in $M.$

In Section \ref{dist from N sect}, we exploit Lemma %
%\ref{sing soon Ric_k Lemma} 
\ref{lem:riccati comparison minus infinity} to prove a rigidity result for
the Jacobi fields of $\Lambda _{N}$ (see Proposition \ref{Jacobi Rigidity
prop}). This allows us to prove, in Section \ref{F structure}, that every
first focal point of $N$ is regular in the sense of \cite{Heda}. With this
it follows rather easily that $F,$ the focal set of $N,$ is a totally
geodesic closed submanifold with focal radius $\frac{\pi }{2}.$ We thus
further the analogy between the pair $\left( N,F\right) $ and the dual sets
in the proof of the Diameter Rigidity Theorem. In particular, we establish,
as in \cite{GrovGrom}, that $F$ (resp. $N$) is the base of a Riemannian
submersion from the unit normal sphere to any point of $N$ (resp. $F$). In
Section \ref{F structure}, we also show that if \textrm{dim}$\left( F\right)
+\mathrm{\dim }\left( N\right) =\dim \left( M\right) -1,$ then $M$ has
constant curvature $1,$ which in particular yields Corollary \ref{hyper
refin cor}.

To show that phenomena like Example \ref{quat exam} do not occur in the
simply connected case, we prove, in Section \ref{simply connected sect},
that our focal set $F$ is \emph{very regular} in the sense of Hebda (\cite%
{Heda}). This allows us to appeal to Theorem 3.1 in \cite{Heda} and
conclude, in Theorem \ref{Hebda 3.1}, that $M$ is the union of two disk
bundles. Using this we prove that if the codimension of $F$ (resp. $N$) is $%
\geq 3,$ then $N$ (resp. $F$) is simply connected; hence the fibers of the
Riemannian submersion to $N$ (resp. $F$) are connected.

All of the above allows us to complete the proof of Theorem \ref{inter Ricci
Rigidity thm} along the lines of the proof of the Diameter Rigidity Theorem.
In the sectional curvature case, the argument can be concluded more rapidly.
We prove that the diameter of the universal cover of $M$ is $\geq \frac{\pi 
}{2},$ and appeal to the Diameter Rigidity Theorem, after making a further
topological argument that rules out exotic spheres and nonunit metrics on $%
\mathbb{S}^{n}.$ We give the details of this in Section \ref{M is not a
sphere}. In Section \ref{Rigid and Inter ric sect}, we complete the proof of
Theorem \ref{inter Ricci Rigidity thm} for intermediate Ricci curvature.

\section{The Distance from $N$}

\label{dist from N sect}

With the exception of Proposition \ref{Jacobi Rigidity prop}, we assume
throughout Sections \ref{dist from N sect}--\ref{Rigid and Inter ric sect}
that $M$ is a complete Riemannian manifold with $\Ric_{k}\geq k,$ and that $%
N $ is a connected, closed submanifold of $M$ of dimension at least $k$ and
focal radius $\frac{\pi }{2}.$

In this section, we apply Lemma \ref{lem:riccati comparison minus infinity} 
%\ref{sing soon Ric_k Lemma}
to prove Proposition \ref{Jacobi Rigidity prop}, which says, among other
things, that the radial sectional curvatures from $N$ are all $\geq 1.$

We start by reviewing the notion of horizontally homothetic submersions.

\begin{definition}
(\cite{BaEells}, \cite{BaWo}) A submersion $\pi :M\longrightarrow B$ of
Riemannian manifolds is called horizontally homothetic if and only if there
is a smooth function $\lambda :M\longrightarrow \left( 0,\infty \right) $
with vertical gradient so that for all horizontal vectors $x$ and $y$%
\begin{equation*}
\lambda ^{2}\left\langle x,y\right\rangle _{M}=\left\langle D\pi \left(
x\right) ,D\pi \left( y\right) \right\rangle _{B}.
\end{equation*}
\end{definition}

We also use the following result from \cite{OW}.

\begin{proposition}
\label{HHS is Riem prop}Let $\pi :M\longrightarrow B$ be a horizontally
homothetic submersion with dilation $\lambda $ and let $r$ be a regular
value of $\lambda $ so that $\lambda ^{-1}\left( r\right) $ is nonempty.
Then 
\begin{equation*}
\pi |_{\lambda ^{-1}\left( r\right) }:\left( \lambda ^{-1}\left( r\right)
,\left\langle \cdot ,\cdot \right\rangle _{M}\right) \longrightarrow \left(
B,\frac{1}{\lambda \left( r\right) ^{2}}\left\langle \cdot ,\cdot
\right\rangle _{B}\right)
\end{equation*}%
is a Riemannian submersion$.$
\end{proposition}

For a unit speed geodesic $\gamma _{v}$ that leaves $N$ orthogonally at time 
$0,$ we set\addtocounter{algorithm}{1} 
\begin{eqnarray}
\mathcal{Z}_{N} &\equiv &\left\{ J|J\left( 0\right) =0\text{, }J^{\prime
}\left( 0\right) \perp \mathrm{span}\left\{ T_{\gamma \left( 0\right) }N%
\text{,}\gamma ^{\prime }\left( 0\right) \right\} \right\}  \notag \\
\mathcal{T}_{N} &\equiv &\left\{ J|J\left( 0\right) \in T_{\gamma \left(
0\right) }N\text{ and }J^{\prime }\left( 0\right) =\mathrm{S}_{v}J\left(
0\right) \right\} ,\text{ and}  \label{Lambda dfn} \\
\Lambda _{N} &\equiv &\mathcal{Z}_{N}\oplus \mathcal{T}_{N},  \notag
\end{eqnarray}%
where $\mathrm{S}_{v}$ is the shape operator of $N$ determined by $v,$ that
is, $\mathrm{S}_{v}:T_{\gamma _{v}\left( 0\right) }N\longrightarrow
T_{\gamma _{v}\left( 0\right) }N,$ is $\left( \nabla _{\cdot }v\right)
^{TN}. $

Our first consequence of $\mathrm{FocalRadius}\left( N\right) =\frac{\pi }{2}
$ holds even if $N$ is not closed, and it only requires that the radial
intermediate Ricci curvatures from $N$ are $\geq k\cdot \kappa .$ Theorem %
\ref{Intermeadiate Ricci Thm} and Remark \ref{radial rem} imply that such an 
$N$ is totally geodesic.

\begin{proposition}
\label{Jacobi Rigidity prop}Let $M$ be a complete Riemannian $n$--manifold,
and let $N$ be a submanifold of $M$ with focal radius $\frac{\pi }{2}$ and $%
\dim \left( N\right) \geq k.$ Suppose that along each unit speed geodesic $%
\gamma :\left[ 0,\frac{\pi }{2}\right] \longrightarrow M$ that leaves $N$
orthogonally at time $0$ we have $\Ric_{k}\left( \dot{\gamma},\cdot \right)
\geq k\cdot \kappa ,$ that is 
\begin{equation*}
\displaystyle\sum\limits_{i=1}^{k}\mathrm{sec}\left( \dot{\gamma}%
,E_{i}\right) \geq k\cdot \kappa ,
\end{equation*}%
for any orthonormal set $\left\{ \dot{\gamma},E_{1},\ldots ,E_{k}\right\} .$

\noindent 1. All $J\in \mathcal{T}_{N}$ have the form $J\left( t\right)
=\cos tE$ where $E$ is a parallel field along $\gamma .$

\noindent 2. $\mathcal{Z}_{N}\oplus \mathcal{T}_{N}$ is a parallel,
orthogonal splitting along $\left[ 0,\frac{\pi }{2}\right] .$

\noindent 3. Let $g^{\ast }$ be the metric on $\mathrm{reg}_{N}\subset \nu
\left( N\right) $ obtained from pulling back $\left( M,g\right) $ via the
normal exponential map, and let $\pi :\mathrm{reg}_{N}\longrightarrow N$ be
the projection of the normal bundle. Then with respect to $g^{\ast }$, $\pi $
is a horizontally homothetic submersion with scaling function $\cos (\mathrm{%
dist}(N_{0},\cdot ))$, where $N_{0}$ is the $0$--section of the normal
bundle, $\nu \left( N\right) .$

\noindent 4. If $c:I\longrightarrow N$ is a unit speed geodesic in $N,$ and $%
V$ is a parallel normal unit field along $c,$ then 
\begin{equation*}
\Phi :I\times \left( 0,\frac{\pi }{2}\right) \longrightarrow M,\text{\hspace{%
0.6in}}\Phi \left( s,t\right) =\exp _{c\left( s\right) }^{\perp }\left(
tV\left( s\right) \right)
\end{equation*}%
is a totally geodesic immersion whose image has constant curvature $1.$

\noindent 5. With respect to $g^{\ast },$ every plane tangent to $\mathrm{reg%
}_{N}\setminus N_{0}$ that contains $X\equiv \mathrm{\mathop{\rm grad}}%
\left\{ \mathrm{dist}\left( N_{0},\cdot \right) \right\} $ has sectional
curvature $\geq 1.$
\end{proposition}

\begin{proof}
Parts 1 and 2 follow from 
%the version of Part 3 of Lemma \ref{sing soon Ric_k Lemma}
the special case of Lemma \ref{lem:riccati comparison minus infinity} when $%
\kappa =1,$ $t_{0}=0,$ $t_{\max }=\frac{\pi }{2},$ and $\tilde{\lambda}%
_{1}=\cot \left( t+\frac{\pi }{2}\right) $ (cf. also Theorem B in \cite%
{GumWilh}).

For Part 3, we let $\mathcal{Z}_{N}^{\ast }$ and $\mathcal{T}_{N}^{\ast }$
be the pullbacks of $\mathcal{Z}_{N}$ and $\mathcal{T}_{N}$ to $\mathrm{reg}%
_{N}$ via $\exp _{N}^{\perp }.$ Observe that by Part 2, 
\begin{equation*}
\mathrm{\mathop{\rm grad}}\left\{ \mathrm{dist}\left( N_{0},\cdot \right)
\right\} \oplus \mathcal{Z}_{N}^{\ast }\oplus \mathcal{T}_{N}^{\ast }
\end{equation*}%
is an orthogonal splitting of $T\mathrm{reg}_{N}$ with respect to $g^{\ast
}. $ Since the vertical space of $\pi $ is spanned by $\mathrm{%
\mathop{\rm
grad}}\left\{ \mathrm{dist}\left( N_{0},\cdot \right) \right\} \oplus 
\mathcal{Z}_{N}^{\ast },$ the horizontal space of $\pi $ with respect to $%
g^{\ast }$ is spanned by $\mathcal{T}_{N}^{\ast }.$ Since the fields of $%
\mathcal{T}_{N}^{\ast }$ come from variations of geodesics that leave $N_{0}$
orthogonally, they are $\pi $--basic. Part 3 follows by combining this with
Part 1.

For Part 4, observe that Part 1 gives us that $\Phi $ is an immersion. Let $%
\mathrm{II}$ be the second fundamental form of $\mathrm{image}\left( \Phi
\right) .$ By construction, $\frac{\partial \Phi }{\partial t}$ is a
geodesic field so $\mathrm{II}\left( \frac{\partial \Phi }{\partial t},\frac{%
\partial \Phi }{\partial t}\right) =0.$ It follows from Part 1 that $\mathrm{%
II}\left( \frac{\partial \Phi }{\partial t},\frac{\partial \Phi }{\partial s}%
\right) =0.$

To see that $\mathrm{II}\left( \frac{\partial \Phi }{\partial s},\frac{%
\partial \Phi }{\partial s}\right) =0,$ we note that, since $\frac{\partial
\Phi }{\partial t}$ is normal to $\exp _{N}^{\perp }\left( S\left(
N_{0},r\right) \right) ,$ it suffices to verify that %
\addtocounter{algorithm}{1} 
\begin{equation}
g\left\langle \mathrm{II}\left( \frac{\partial \Phi }{\partial s},\frac{%
\partial \Phi }{\partial s}\right) ,Z\right\rangle =0  \label{tang enough}
\end{equation}%
for all $Z\in T\exp _{N}^{\perp }\left( S\left( N_{0},r\right) \right) $
that are normal to the image of $\Phi .$ Since $\exp _{N}^{\perp }:\left( 
\mathrm{reg}_{N},g^{\ast }\right) \longrightarrow \left( M,g\right) $ is a
local isometry, to prove \eqref{tang enough}, it suffices to do the
corresponding calculation in $\left( \mathrm{reg}_{N},g^{\ast }\right) .$
From Part 3 we have that the restriction of $\pi :\left( \mathrm{reg}%
_{N},g^{\ast }\right) \longrightarrow \left( N,\frac{1}{\cos \left( t\right)
^{2}}g\right) $ to the $t$--level set of $\mathrm{dist}(N_{0},\cdot )$ is a
Riemannian submersion. Let $\widetilde{\frac{\partial \Phi }{\partial s}}$
be a lift of $\frac{\partial \Phi }{\partial s}$ to $\mathrm{reg}_{N}$ via $%
\exp _{N}^{\perp }.$ Then $\widetilde{\frac{\partial \Phi }{\partial s}}$ is
a $\pi $--basic horizontal, geodesic field, so $\mathrm{II}\left( \frac{%
\partial \Phi }{\partial s},\frac{\partial \Phi }{\partial s}\right) =\left(
D\exp _{N}^{\perp }\right) \left( \mathrm{II}\left( \widetilde{\frac{%
\partial \Phi }{\partial s}},\widetilde{\frac{\partial \Phi }{\partial s}}%
\right) \right) \equiv 0.$ Hence the $\mathrm{image}\left( \Phi \right) $ is
totally geodesic. It follows from Part 1 that $\mathrm{image}\left( \Phi
\right) $ has constant curvature $1.$

To prove Part 5, we let $\left\{ J_{1}^{\ast },\ldots ,J_{k-1}^{\ast
}\right\} $ be any $k-1$ linearly independent Jacobi fields in $\mathcal{T}%
_{N}^{\ast }$. It follows from Part 1 that for all $i,$ 
\begin{equation*}
\mathrm{sec}\left( \mathrm{\mathop{\rm grad}}\left\{ \mathrm{dist}\left(
N_{0},\cdot \right) \right\} ,J_{i}^{\ast }\right) \equiv 1.
\end{equation*}%
Together with Part 2 and our hypothesis that $\Ric_{k}\left( \dot{\gamma}%
,\cdot \right) \geq k$, we conclude that for all $J^{\ast }\in $ $\mathcal{Z}%
_{N}^{\ast },$ $\mathrm{sec}\left( \mathrm{\mathop{\rm grad}}\left\{ \mathrm{%
dist}\left( N_{0},\cdot \right) \right\} ,J^{\ast }\right) \geq 1.$

It follows from Part 2 that $\mathcal{Z}_{N}^{\ast }\oplus \mathcal{T}%
_{N}^{\ast }$ is a splitting of $\Lambda _{N_{0}}$ into orthogonal,
invariant subspaces for $R\left( \cdot ,\mathrm{\mathop{\rm grad}}\left\{ 
\mathrm{dist}\left( N_{0},\cdot \right) \right\} \right) \mathrm{%
\mathop{\rm
grad}}\left\{ \mathrm{dist}\left( N_{0},\cdot \right) \right\} .$ So 
\begin{equation*}
\mathrm{sec}\left( \mathrm{\mathop{\rm grad}}\left\{ \mathrm{dist}\left(
N_{0},\cdot \right) \right\} ,Y\right) \geq 1
\end{equation*}%
for all vectors $Y$ orthogonal to $\mathrm{\mathop{\rm grad}}\left\{ \mathrm{%
dist}\left( N_{0},\cdot \right) \right\} .$
\end{proof}

\section{The Structure of the Focal Set\label{F structure}}

This section begins with a review of Hebda's notion of regular tangent focal
points that generalizes a notion of Warner for conjugate points (\cite{Heda}%
, \cite{Warn2}). We next exploit the rigidity in Proposition \ref{Jacobi
Rigidity prop} to show that every tangent focal point at time $\frac{\pi }{2}
$ is regular. This allows us to apply a result of Hebda and conclude that
our focal set $F\equiv \exp _{N}^{\perp }\left( S\left( N_{0},\frac{\pi }{2}%
\right) \right) $ is a smooth submanifold of $M.$ The rigid structure also
yields that $F$ has focal radius $\frac{\pi }{2}.$ We then further the
analogy between the pair $\left( N,F\right) $ and the dual sets in the proof
of the Diameter Rigidity Theorem by showing that $F$ (resp. $N$) is the base
of a Riemannian submersion from the unit normal sphere to any point of $N$
(resp. $F$).

\begin{definition}
(\cite{Heda}, cf \cite{Warn2}) A tangent focal point $v\in \nu \left(
N\right) $ is called regular if and only if there is a neighborhood $U$ of $%
v $ so that every ray in $\nu \left( N\right) $ that intersects $U$ has at
most one tangent focal point in $U,$ not counting multiplicities. Otherwise $%
v$ is called singular.
\end{definition}

Continuity of the curvature tensor implies that every $v\in \nu \left(
N\right) $ has a neighborhood $U$ so that every ray meeting $U$ has the same
number of tangent focal points, counting multiplicities. So if $v$ is a
regular tangent focal point, then every ray $tu$ in $\nu \left( N\right) $
that intersects $U$ has exactly one focal point $t_{0}u$, and the
multiplicities of $t_{0}u$ and $v$ coincide$.$ Thus regular tangent focal
points have locally maximal order. Using this and ideas of \cite{Warn2},
Hebda showed the following.

\begin{theorem}
(\cite{Heda}, cf \cite{Warn2}) The set of regular tangent focal points is a
smooth codimension 1 submanifold of $\nu \left( N\right) $ that is an open,
dense subset of the set of all tangent focal points.
\end{theorem}

On $\mathrm{reg}_{N}\subset \nu \left( N\right) ,$ we set 
\begin{equation*}
X\equiv \mathop{\rm grad}\left( \mathrm{dist}\left( N_{0},\cdot \right)
\right) .
\end{equation*}

Along a fixed geodesic, focal points are isolated, so it follows that the
set of regular, first-tangent focal points is an open, dense subset of the
set of first-tangent focal points. It follows from the Gauss Lemma that $%
\ker \left( D\exp _{N}^{\perp }\right) _{X}\perp X.$ Since the first-tangent
focal set of our $N\subset M$ is $S\left( N_{0},\frac{\pi }{2}\right) ,$ it
follows that $\ker \left( D\exp _{N}^{\perp }\right) \subset TS\left( N_{0},%
\frac{\pi }{2}\right) .$ Combining this with the Rank Theorem we get

\begin{corollary}
\label{F_reg is open dense cor}Let $\tilde{F}_{\mathrm{reg}}$ be the set of
regular first-tangent focal points, and let 
\begin{equation*}
F_{\mathrm{reg}}\equiv \exp _{N}^{\perp }\left( \tilde{F}_{\mathrm{reg}%
}\right) .
\end{equation*}%
Then $F_{\mathrm{reg}}$ is a smooth submanifold of $M$ that is open and
dense inside of $F.$
\end{corollary}

\begin{lemma}
\label{dim drop lem}Let $v\in \nu \left( N\right) $ be a singular tangent
focal point. For every neighborhood $U$ of $v,$ there is a regular tangent
focal point $w\in U$ so that 
\begin{equation*}
\mathrm{\dim }\left( \ker \left( D\exp _{N}^{\perp }\right) _{w}\right) <%
\mathrm{\dim }\left( \ker \left( D\exp _{N}^{\perp }\right) _{v}\right) .
\end{equation*}
\end{lemma}

\begin{proof}
Let $U$ be any neighborhood of $v.$ Replacing $U$ with a possibly smaller
neighborhood, we may assume that the total multiplicity of the focal points
on each ray that intersects $U$ is constant, and that the ray $tv$ contains
only one focal point in $U$. Since $v$ is singular, $U$ contains a ray with
more than one focal point $w_{1}\neq w_{2}$, which by hypothesis is not the
ray through $v.$ Since the multiplicity of the focal points in $tw_{1}\cap U$
and $tv\cap U$ is the same, it follows that\addtocounter{algorithm}{1} 
\begin{equation}
\mathrm{\dim }\left( \ker \left( D\exp _{N}^{\perp }\right) _{w_{1}}\right) <%
\mathrm{\dim }\left( \ker \left( D\exp _{N}^{\perp }\right) _{v}\right) .
\label{dropping dim ker}
\end{equation}%
It might be that $w_{1}$ is not regular; however, since (\ref{dropping dim
ker}) holds for some $w_{1}$ in any neighborhood of $v,$ by repeating this
argument a finite number of times, we get the desired conclusion.
\end{proof}

Let $\gamma $ be a unit speed geodesic that leaves $N$ orthogonally at time $%
0$ with $\gamma \left( \frac{\pi }{2}\right) \in F_{\mathrm{reg}}.$ Recall
that the elements of $\Lambda _{N}$ are called $N$--Jacobi fields. We set%
\begin{eqnarray*}
\mathcal{Z} &\equiv &\left\{ \left. J\in \Lambda _{N}\text{ }\right\vert 
\text{ }J\left( 0\right) =J\left( \frac{\pi }{2}\right) =0\right\} , \\
\mathcal{T}_{N} &\equiv &\left\{ J\in \Lambda _{N}|J\left( 0\right) \in
T_{\gamma \left( 0\right) }N\text{ and }J^{\prime }\left( 0\right)
=S_{\gamma ^{\prime }\left( 0\right) }\left( J\left( 0\right) \right)
\right\} ,\text{ and} \\
\mathcal{T}_{F_{\mathrm{reg}}} &\equiv &\left\{ J|\text{ }J\left( \frac{\pi 
}{2}\right) \in T_{\gamma \left( \frac{\pi }{2}\right) }F_{\mathrm{reg}}%
\text{ and }J^{\prime }\left( \frac{\pi }{2}\right) =-S\left( J\left( \frac{%
\pi }{2}\right) \right) \right\} ,
\end{eqnarray*}%
where $S_{\gamma ^{\prime }\left( 0\right) }$ in the definition of $\mathcal{%
T}_{N}$ is the shape operator of $N$ and $S$ in the definition of $\mathcal{T%
}_{F_{\mathrm{reg}}}$ is the Riccati operator of $\Lambda _{N}.$ The next
lemma shows that the $S$ in the definition of $\mathcal{T}_{F_{\mathrm{reg}%
}} $ is also the shape operator of $F_{\mathrm{reg}}$ with respect to $%
\gamma ^{\prime }\left( \frac{\pi }{2}\right) .$

\begin{lemma}
\label{duality lemma copy(1)}For $\gamma $ as above:

\noindent 1. $\gamma ^{\prime }\left( \frac{\pi }{2}\right) \in \nu _{\gamma
\left( \frac{\pi }{2}\right) }F_{\mathrm{reg}}.$

\noindent 2. The $N$--Jacobi fields along $\gamma $ are the $F_{\mathrm{reg}%
} $--Jacobi fields along 
\begin{equation*}
\gamma ^{-1}:t\mapsto \gamma \left( \frac{\pi }{2}-t\right) .
\end{equation*}

\noindent 3. The subspaces $\mathcal{T}_{N}$ and $\mathcal{T}_{F_{\mathrm{reg%
}}}$ are rigid, that is, 
\begin{eqnarray*}
\mathcal{T}_{N} &=&\left\{ \cos t\,E|\text{ }E\text{ is parallel and tangent
to }N\text{ at time }0\right\} ,\text{ and} \\
\mathcal{T}_{F_{\mathrm{reg}}} &=&\left\{ \sin t\,E|\text{ }E\text{ is
parallel and tangent to }F_{\mathrm{reg}}\text{ at time }\frac{\pi }{2}%
\right\} .
\end{eqnarray*}

\noindent 4. Writing $\Lambda _{N}$ for the $N$--Jacobi fields along $\gamma
,$ we have orthogonal splittings 
\begin{eqnarray*}
\Lambda _{N} &=&\mathcal{T}_{N}\oplus \mathcal{T}_{F_{\mathrm{reg}}}\oplus 
\mathcal{Z}\text{ and} \\
\mathcal{Z}_{N} &=&\mathcal{T}_{F_{\mathrm{reg}}}\oplus \mathcal{Z},
\end{eqnarray*}%
where $\mathcal{Z}_{N}$ is as in Equation \eqref{Lambda dfn}.
\end{lemma}

\begin{proof}
Part 1 is a consequence of the Gauss Lemma and the fact that $F\equiv \exp
_{N}^{\perp }\left( S\left( N_{0},\frac{\pi }{2}\right) \right) $. The space 
$\Lambda _{N}$ of $N$--Jacobi fields along $\gamma $ are precisely the
variation fields of variations by geodesics that leave $N$ orthogonally at
time $0.$ Similarly, the space $\Lambda _{F_{\mathrm{reg}}}$ of $F_{\mathrm{%
reg}}$--Jacobi fields along $\gamma $ are precisely the variation fields of
variations by geodesics that arrive at $F_{\mathrm{reg}}$ orthogonally at
time $\frac{\pi }{2}.$ It follows from Part 1 that $\Lambda _{N}\subset
\Lambda _{F_{\mathrm{reg}}}.$ Since $\dim \left( \Lambda _{N}\right)
=n-1=\dim \left( \Lambda _{F_{\mathrm{reg}}}\right) ,$ $\Lambda _{N}=\Lambda
_{F_{\mathrm{reg}}}.$ This proves Part 2.

Since $\gamma $ has no focal points for $N$ on $\left( 0,\frac{\pi }{2}%
\right) ,$ it follows from Part 2 that $\gamma ^{-1}\left( t\right) =\gamma
\left( \frac{\pi }{2}-t\right) $ has no focal points for $F_{\mathrm{reg}}$
on $\left( 0,\frac{\pi }{2}\right) .$ By Part 5 of Proposition \ref{Jacobi
Rigidity prop}, all the radial sectional curvatures along $\gamma $ are $%
\geq 1.$ Thus Parts 3 and 4 follow from Parts 1 and 2 of Proposition \ref%
{Jacobi Rigidity prop} and the fact that $\gamma ^{-1}\left( t\right)
=\gamma \left( \frac{\pi }{2}-t\right) $ has no focal points for $F_{\mathrm{%
reg}}$ on $\left( 0,\frac{\pi }{2}\right) .$
\end{proof}

\begin{lemma}
\label{F equals F_reg lemma}$F_{\mathrm{reg}}=F.$
\end{lemma}

\begin{proof}
We set 
\begin{equation*}
F_{\mathrm{sng}}\equiv F\setminus F_{\mathrm{reg}},
\end{equation*}%
and suppose, by way of contradiction, that $F_{\mathrm{sng}}\neq \emptyset .$

Let $\gamma _{\mathrm{reg}}$ and $\gamma _{\mathrm{sng}}$ be geodesics that
leave $N$ orthogonally at time $0$ with 
\begin{equation*}
\gamma _{\mathrm{reg}}\left( \frac{\pi }{2}\right) \in F_{\mathrm{reg}}\text{
and }\gamma _{\mathrm{sng}}\left( \frac{\pi }{2}\right) \in F_{\mathrm{sng}}.
\end{equation*}

The idea of the proof is to examine how the splitting $\mathcal{Z}_{N}=%
\mathcal{T}_{F_{\mathrm{reg}}}\oplus \mathcal{Z}$ behaves as a sequence of $%
\gamma _{\mathrm{reg}}$'s approaches $\gamma _{\mathrm{sng}}.$ In
particular, by Lemma \ref{duality lemma copy(1)}, $\mathcal{T}_{F_{\mathrm{%
reg}}}$ is spanned by constant curvature $1$ Jacobi fields. By continuity, $%
\gamma _{\mathrm{sng}}$ inherits such a family, and this forces $\frac{\pi }{%
2}\gamma _{\mathrm{sng}}^{\prime }\left( 0\right) $ to actually be regular.
The details follow.

By appealing to Lemma \ref{dim drop lem}, we can assume that %
\addtocounter{algorithm}{1} 
\begin{equation}
\mathrm{\dim }\left( \ker \left( D\exp _{N}^{\perp }\right) _{\frac{\pi }{2}%
\gamma _{\mathrm{reg}}^{\prime }\left( 0\right) }\right) <\mathrm{\dim }%
\left( \ker \left( D\exp _{N}^{\perp }\right) _{\frac{\pi }{2}\gamma _{%
\mathrm{sng}}^{\prime }\left( 0\right) }\right) .  \label{kernel goes donw}
\end{equation}

For either $\gamma _{\mathrm{reg}}$ or $\gamma _{\mathrm{sng}}$ we have the
four spaces of Jacobi fields, $\Lambda _{N},$ $\mathcal{T}_{N},$ $\mathcal{Z}%
_{N},$ and $\mathcal{Z}.$ We will distinguish the versions of the spaces
along $\gamma _{\mathrm{reg}}$ from those along $\gamma _{\mathrm{sng}}$
with the superscripts $^{\mathrm{reg}}$ and $^{\mathrm{sng}}.$ When no
superscript is present, the statement applies to either case.

For either $\gamma _{\mathrm{reg}}$ or $\gamma _{\mathrm{sng}},$ %
\addtocounter{algorithm}{1} 
\begin{equation}
\ker \left( D\exp _{N}^{\perp }\right) _{\frac{\pi }{2}\gamma ^{\prime
}\left( 0\right) }=\left\{ \left. J\left( 0\right) \text{ }\right\vert \text{
}J\in \mathcal{T}_{N}\right\} \oplus \left\{ \left. J^{\prime }\left(
0\right) \text{ }\right\vert \text{ }J\in \mathcal{Z}\right\} .
\label{ker Dexp-2}
\end{equation}%
Thus 
\begin{eqnarray*}
\mathrm{\dim }\left( \ker \left( D\exp _{N}^{\perp }\right) _{\frac{\pi }{2}%
\gamma ^{\prime }\left( 0\right) }\right) &=&\dim \left( \mathcal{T}%
_{N}\right) +\dim \left( \mathcal{Z}\right) \\
&=&\dim \left( N\right) +\dim \left( \mathcal{Z}\right) .
\end{eqnarray*}%
Since $\mathrm{\dim }\left( \ker \left( D\exp _{N}^{\perp }\right) _{\frac{%
\pi }{2}\gamma _{\mathrm{sng}}^{\prime }\left( 0\right) }\right) >\mathrm{%
\dim }\left( \ker \left( D\exp _{N}^{\perp }\right) _{\frac{\pi }{2}\gamma _{%
\mathrm{reg}}^{\prime }\left( 0\right) }\right) ,$ the dimensions of $%
\mathcal{Z}^{\mathrm{reg}}$ and $\mathcal{Z}^{\mathrm{sng}}$ satisfy%
\addtocounter{algorithm}{1} 
\begin{equation}
\mathrm{\dim }\left( \mathcal{Z}^{\mathrm{sng}}\right) >\mathrm{\dim }\left( 
\mathcal{Z}^{\mathrm{reg}}\right) .  \label{relate the Zs}
\end{equation}

Along $\gamma _{\mathrm{reg}},$ Lemma \ref{duality lemma copy(1)} gives us
an orthogonal splitting\addtocounter{algorithm}{1} 
\begin{equation}
\mathcal{Z}_{N}^{\mathrm{reg}}=\mathcal{T}_{F_{\mathrm{reg}}}^{\mathrm{reg}%
}\oplus \mathcal{Z}^{\mathrm{reg}},  \label{Z_N^reg split}
\end{equation}%
where \addtocounter{algorithm}{1} 
\begin{equation}
\mathcal{T}_{F_{\mathrm{reg}}}^{\mathrm{reg}}=\left\{ \sin t\,E|\text{ }E%
\text{ is parallel and tangent to }F_{\mathrm{reg}}\text{ at time }\frac{\pi 
}{2}\right\} .  \label{T_F reg eqn}
\end{equation}%
Combined with Inequality \eqref{relate the Zs}, this gives %
\addtocounter{algorithm}{1} 
\begin{eqnarray}
\mathrm{\dim }\left( \mathcal{Z}^{\mathrm{sng}}\right) &>&\mathrm{\dim }%
\left( \mathcal{Z}^{\mathrm{reg}}\right)  \notag \\
&=&\mathrm{\dim }\left( \mathcal{Z}_{N}^{\mathrm{reg}}\right) -\dim \left( 
\mathcal{T}_{F_{\mathrm{reg}}}^{\mathrm{reg}}\right) ,\text{ by \ref{Z_N^reg
split}}  \notag \\
&=&\left( n-1\right) -\dim \left( N\right) -\mathrm{\dim }\left( \mathcal{T}%
_{F_{\mathrm{reg}}}^{\mathrm{reg}}\right) .  \label{rel Zs-2}
\end{eqnarray}

Note that $\gamma _{\mathrm{sng}}$ is a limit of $\gamma _{\mathrm{reg}}$'s
that satisfy (\ref{kernel goes donw}). Further note that a $J\in \mathcal{T}%
_{F_{\mathrm{reg}}}^{\mathrm{reg}}$ together with $\gamma _{\mathrm{reg}%
}^{\prime }$ spans a plane of constant curvature $1.$ Thus by continuity, $%
\mathcal{Z}_{N}^{\mathrm{sng}}$ contains a subspace $\mathcal{T}^{\mathrm{sng%
}}$ of the form 
\begin{equation*}
\mathcal{T}^{\mathrm{sng}}=\left\{ \sin t\,E|\text{ }E\text{ is parallel}%
\right\} \subset \mathcal{Z}_{N}^{\mathrm{sng}}\setminus \mathcal{Z}^{%
\mathrm{sng}}
\end{equation*}%
with \addtocounter{algorithm}{1} 
\begin{equation}
\mathrm{\dim }\left( \mathcal{T}^{\mathrm{sng}}\right) =\mathrm{\dim }\left( 
\mathcal{T}_{F_{\mathrm{reg}}}^{\mathrm{reg}}\right) .  \label{Ts same eqn}
\end{equation}%
Moreover, by Remark \ref{iniital sing remark}, Part 2 of Lemma \ref{sing
soon Ric_k Lemma}, and Part 5 of Proposition \ref{Jacobi Rigidity prop}, $%
\Lambda _{N}^{\mathrm{sng}}$ splits orthogonally with one factor being $%
\mathcal{T}^{\mathrm{sng}}.$ Since $\mathcal{T}^{\mathrm{sng}}$ is a
subspace of $\mathcal{Z}_{N}^{\mathrm{sng}},$ we get %
\addtocounter{algorithm}{1} 
\begin{equation}
\mathcal{Z}_{N}^{\mathrm{sng}}=\mathcal{T}^{\mathrm{sng}}\oplus \mathcal{U}^{%
\mathrm{sng}},  \label{Z_N^sing split}
\end{equation}%
where $\mathcal{U}^{\mathrm{sng}}$ is a space of Jacobi fields in $\mathcal{Z%
}_{N}^{\mathrm{sng}}$ that is orthogonal to $\mathcal{T}^{\mathrm{sng}}$
throughout $\left( 0,\frac{\pi }{2}\right) .$

The splitting \eqref{Z_N^sing split} combined with $\mathcal{T}^{\mathrm{sng}%
}=\left\{ \sin t\,E|\text{ }E\text{ is parallel}\right\} $ gives that $%
\mathcal{Z}^{\mathrm{sng}}$ is a subspace of $\mathcal{U}^{\mathrm{sng}},$
so 
\begin{eqnarray*}
\dim \left( \mathcal{Z}^{\mathrm{sng}}\right) &\leq &\dim \left( \mathcal{U}%
^{\mathrm{sng}}\right) \\
&=&\dim \left( \mathcal{Z}_{N}^{\mathrm{sng}}\right) -\mathrm{\dim }\left( 
\mathcal{T}_{F_{\mathrm{reg}}}^{\mathrm{reg}}\right) ,\text{ by \eqref{Ts
same eqn} and \eqref{Z_N^sing split}} \\
&=&\left( n-1\right) -\dim \left( N\right) -\mathrm{\dim }\left( \mathcal{T}%
_{F_{\mathrm{reg}}}^{\mathrm{reg}}\right) .
\end{eqnarray*}%
Since this contradicts Inequality \eqref{rel Zs-2}, the result is proven.%
{\huge \ }
\end{proof}

\begin{lemma}
\label{Focal F is pi/2 lemma}$F$ is a totally geodesic closed submanifold of 
$M$ with focal radius $\frac{\pi }{2}.$
\end{lemma}

\begin{proof}
Since $F_{\mathrm{reg}}=F,$ it is a submanifold. Since $F=\exp _{N}^{\perp
}\left( S\left( N_{0},\frac{\pi }{2}\right) \right) ,$ it is closed. It
follows from Part 2 of Lemma \ref{duality lemma copy(1)} that $\Lambda
_{N}=\Lambda _{F}.$ Therefore the focal radius of $F$ is $\frac{\pi }{2}.$

We have $F=F_{\mathrm{reg}},$ so from Part 3 of Lemma \ref{duality lemma
copy(1)}, 
\begin{equation*}
\mathcal{T}_{F}=\left\{ \sin t\,E|\text{ }E\text{ is parallel and tangent to 
}F\text{ at time }\frac{\pi }{2}\right\} .
\end{equation*}%
In particular, for $J\in \mathcal{T}_{F},$ $J^{\prime }\left( \frac{\pi }{2}%
\right) =0.$ So $F$ is totally geodesic.
\end{proof}

The next result will lead to the spherical rigidity portion of the
conclusion of Theorem \ref{inter Ricci Rigidity thm}, and also gives us
Corollary \ref{hyper refin cor}.

\begin{theorem}
\label{const curv thm}$M$ has constant curvature $1$ if either of the
following holds:

\noindent 1. $F$ is not connected.

\noindent 2. \textrm{dim}$\left( F\right) +\mathrm{\dim }\left( N\right)
=\dim \left( M\right) -1.$
\end{theorem}

\begin{proof}
We have that $\exp ^{\perp }:S\left( N_{0},\frac{\pi }{2}\right)
\longrightarrow F$ is onto. So if $F$ is not connected, then $S\left( N_{0},%
\frac{\pi }{2}\right) $ is not connected, and it follows that $N$ is
codimension 1 and has a trivial normal bundle. Since 
\begin{equation*}
\mathrm{dim}\left( F\right) +\mathrm{\dim }\left( N\right) \leq \dim \left(
\Lambda _{N}\right) =\dim \left( M\right) -1,
\end{equation*}%
to prove Part 1, it is enough to prove Part 2.

In general, we have an orthogonal splitting of $\Lambda _{N}=\mathcal{T}%
_{N}\oplus \mathcal{T}_{F}\oplus \mathcal{Z}$ along any one of our normal
geodesics$.$ Since $\mathrm{\dim }\left( \mathcal{T}_{F}\right) =\mathrm{%
\dim }\left( F\right) $, $\mathrm{\dim }\left( \mathcal{T}_{N}\right) =%
\mathrm{\dim }\left( N\right) ,$ and \textrm{dim}$\left( F\right) +\mathrm{%
\dim }\left( N\right) =\dim \left( M\right) -1,$ $\mathcal{Z}=0,$ and our
geodesic is spanned by constant curvature $1$ Jacobi fields. Moreover, 
\begin{equation*}
\mathcal{T}_{N}=\mathcal{Z}_{F}\text{ and }\mathcal{T}_{F}=\mathcal{Z}_{N}.
\end{equation*}%
Combining this with Part 1 of Proposition \ref{Jacobi Rigidity prop}, it
follows that along a geodesic leaving $F$ orthogonally at time $0,$%
\addtocounter{algorithm}{1} 
\begin{equation}
\mathcal{Z}_{F}=\left\{ \left. \sin tE\text{ }\right\vert \text{ }E\text{ is
parallel}\right\} ,  \label{F metric spheres}
\end{equation}%
and along a geodesic leaving $N$ orthogonally at time $0,$%
\addtocounter{algorithm}{1} 
\begin{equation}
\mathcal{Z}_{N}=\left\{ \left. \sin tE\text{ }\right\vert \text{ }E\text{ is
parallel}\right\} .  \label{Z metric spheres}
\end{equation}%
It follows from Equation \eqref{F metric spheres} that for all $x\in F$ and
all $r\in \left( 0,\frac{\pi }{2}\right) ,$ the intrinsic metrics on %
\addtocounter{algorithm}{1} 
\begin{equation}
S_{F}\left( x,r\right) \equiv \exp _{x}^{\perp }\left\{ \left. v\in \nu
_{x}\left( F\right) \text{ }\right\vert \text{ }\left\vert v\right\vert
=r\right\}  \label{F metric sphere II}
\end{equation}%
are locally isometric to $S^{\mathrm{\dim N}}\left( \sin r\right) ,$ that
is, to the sphere of radius $\sin r$ in $\mathbb{R}^{\mathrm{\dim N}+1}.$
Similarly, it follows that for all $x\in N$ and all $r\in \left( 0,\frac{\pi 
}{2}\right) ,$ the intrinsic metrics on \addtocounter{algorithm}{1} 
\begin{equation}
S_{N}\left( x,r\right) \equiv \exp _{x}^{\perp }\left\{ \left. v\in \nu
_{x}\left( N\right) \text{ }\right\vert \text{ }\left\vert v\right\vert
=r\right\}  \label{N metric sphere II}
\end{equation}%
are locally isometric to $S^{\mathrm{\dim F}}\left( \sin r\right) .$ Since $%
\mathcal{T}_{N}\oplus \mathcal{Z}_{N}$ is an orthogonal splitting, if $%
\gamma $ leaves $N$ orthogonally at time $0,$ then %
\addtocounter{algorithm}{1} 
\begin{equation}
S_{N}\left( \gamma \left( 0\right) ,r\right) \text{ and }S_{F}\left( \gamma
\left( \frac{\pi }{2}\right) ,\frac{\pi }{2}-r\right) \text{ intersect
orthogonally at }\gamma \left( r\right) .  \label{orthog inter}
\end{equation}

Let 
\begin{equation*}
S_{N}\left( r\right) \equiv \exp _{N}^{\perp }\left\{ \left. v\in \nu \left(
N\right) \text{ }\right\vert \text{ }\left\vert v\right\vert =r\right\} ,
\end{equation*}%
and let $\mathrm{II}_{r}$ be the second fundamental form of $S_{N}\left(
r\right) ,$ that is 
\begin{equation*}
\mathrm{II}_{r}\left( U,V\right) \equiv g\left( \nabla _{U}V,\gamma ^{\prime
}\left( r\right) \right) ,
\end{equation*}%
where $\gamma $ leaves $N$ orthogonally at time $0.$

Combining \eqref{F metric spheres}, \eqref{Z metric spheres} and 
\eqref{orthog
inter}, we have for $Y\in \mathcal{T}_{N},$ and $W\in \mathcal{Z}_{N}$, %
\addtocounter{algorithm}{1} 
\begin{eqnarray}
\mathrm{II}_{r}\left( Y,Y\right) &=&\left\vert Y\right\vert ^{2}\tan \left(
r\right)  \notag \\
\mathrm{II}_{r}\left( W,W\right) &=&-\left\vert W\right\vert ^{2}\cot \left(
r\right) ,\text{ and}  \notag \\
\mathrm{II}_{r}\left( Y,W\right) &=&0.  \label{II spherical}
\end{eqnarray}

Now view $\mathbb{S}^{n}$ as a join, $\mathbb{S}^{n}=\mathbb{S}^{\mathrm{%
\dim N}}\ast \mathbb{S}^{\dim \mathrm{F}},$ and let $\tilde{\gamma}$ be a
geodesic that leaves $\mathbb{S}^{\mathrm{\dim N}}$ orthogonally at time $0.$
Setting $\tilde{M}\equiv \mathbb{S}^{n},$ $\tilde{N}\equiv \mathbb{S}^{%
\mathrm{\dim N}},$ and $\tilde{F}\equiv \mathbb{S}^{\dim \mathrm{F}},$
observe that \eqref{F metric sphere II}, \eqref{N metric sphere II}, 
\eqref{orthog
inter}, and \eqref{II spherical} hold with $M$, $N,$ and $F$ replaced by $%
\tilde{M}$, $\tilde{N},$ and $\tilde{F}.$ Observe further that Equations %
\eqref{F metric sphere II}, \eqref{N metric sphere II}, \eqref{orthog inter}%
, and \eqref{II spherical} together with the Gauss, Radial, and
Codazzi-Mainardi Equations (\cite{Pet}) determine the curvature tensor of $%
\tilde{M}\equiv S^{n}.$ Similarly, they determine the curvature tensor of $%
M. $ Thus $M$ has constant curvature 1.
\end{proof}

Throughout the remainder of Part 2, we assume that $F$ is connected and %
\addtocounter{algorithm}{1} 
\begin{equation}
\mathrm{dim}\left( F\right) +\mathrm{\dim }\left( N\right) \leq \dim \left(
M\right) -2.  \label{dim hypoth eqn}
\end{equation}

\begin{lemma}
\noindent 1. Let $x\in N.$ With respect to the constant curvature $1$ metric
on the unit normal sphere, $\nu _{x}^{1}\left( N\right) ,$ the map 
\begin{eqnarray*}
\pi _{x} &:&\nu _{x}^{1}\left( N\right) \longrightarrow F \\
\pi _{x} &:&v\mapsto \exp _{N}\left( \frac{\pi }{2}v\right)
\end{eqnarray*}%
is a Riemannian submersion onto $F.$

\noindent 2. Let $x\in F.$ With respect to the constant curvature $1$ metric
on the unit normal sphere, $\nu _{x}^{1}\left( F\right) ,$ the map 
\begin{eqnarray*}
\pi _{x} &:&\nu _{x}^{1}\left( F\right) \longrightarrow N \\
\pi _{x} &:&v\mapsto \exp _{F}\left( \frac{\pi }{2}v\right)
\end{eqnarray*}%
is a Riemannian submersion onto $N.$
\end{lemma}

\begin{proof}
The proofs are identical, except for notation . We give the details for Part
1.

Let $\gamma $ be a geodesic that leaves $N$ orthogonally at time $0.$ Then %
\addtocounter{algorithm}{1} 
\begin{eqnarray}
T_{\gamma ^{\prime }\left( 0\right) }\left( \nu _{\gamma \left( 0\right)
}^{1}\left( N\right) \right) &=&\left\{ \left. J^{\prime }\left( 0\right) 
\text{ }\right\vert \text{ }J\in \mathcal{Z}_{N}\right\} ,\text{ and}  \notag
\\
D\pi _{\gamma \left( 0\right) }\left( J^{\prime }\left( 0\right) \right)
&=&J\left( \frac{\pi }{2}\right)  \label{Normal fiber eqn}
\end{eqnarray}%
for all $J\in \mathcal{Z}_{N}.$

Since $\mathcal{T}_{F}\subset \mathcal{Z}_{N},$ the splittings $\Lambda _{N}=%
\mathcal{T}_{N}\oplus \mathcal{Z}_{N}=\mathcal{T}_{N}\oplus \mathcal{T}%
_{F}\oplus \mathcal{Z}$ give us an orthogonal splitting 
\begin{equation*}
\mathcal{Z}_{N}=\mathcal{Z}\oplus \mathcal{T}_{F}.
\end{equation*}%
Combined with Equation \eqref{Normal fiber eqn} this gives an orthogonal
splitting \addtocounter{algorithm}{1} 
\begin{equation}
T_{\gamma ^{\prime }\left( 0\right) }\left( \nu _{\gamma \left( 0\right)
}^{1}\left( N\right) \right) =\left\{ \left. J^{\prime }\left( 0\right) 
\text{ }\right\vert \text{ }J\in \mathcal{Z}\right\} \oplus \left\{ \left.
J^{\prime }\left( 0\right) \text{ }\right\vert \text{ }J\in \mathcal{T}%
_{F}\right\}  \label{normal sphere splitting eqn}
\end{equation}%
into the vertical and horizontal spaces, respectively, for $\pi _{\gamma
\left( 0\right) }.$ Since $\mathrm{\dim }\left( \mathcal{T}_{F}\right) =\dim
\left( F\right) ,$ it follows from Equation \eqref{Normal fiber eqn} that $%
\pi _{\gamma \left( 0\right) }$ is a submersion. By Part 1 of Proposition %
\ref{Jacobi Rigidity prop}, with $F$ playing the role of $N,$ the
restriction of $D\pi _{\gamma \left( 0\right) }$ to the second summand in
Equation \eqref{normal sphere splitting eqn} is an isometry. Thus $\pi
_{\gamma \left( 0\right) }$ is a Riemannian submersion.
\end{proof}

\section{The Simply Connected Case\label{simply connected sect}}

Let $\pi :\tilde{M}\longrightarrow M$ be the universal cover of $M$. Then
each component $\pi ^{-1}\left( N\right) $ is a submanifold with focal
radius $\frac{\pi }{2}$ and dimension at least $k.$ In particular, $\tilde{M}
$ contains a closed, connected, embedded submanifold with focal radius $%
\frac{\pi }{2}$ and dimension at least $k.$ So to prove Theorem \ref{inter
Ricci Rigidity thm}, it suffices to consider the case when $M$ is simply
connected and $N$ is connected.

In this section, we will combine our simply connected hypothesis with
Hebda's theorem on \textquotedblleft very regular\textquotedblright\ focal
loci. This will allow us to assert that, topologically, $M$ is the union of
two disk bundles, and the fibers of our Riemannian submersions, 
\begin{eqnarray*}
\pi _{x} &:&\nu _{x}^{1}\left( N\right) \longrightarrow F \\
\pi _{x} &:&v\mapsto \exp _{N}\left( \frac{\pi }{2}v\right)
\end{eqnarray*}%
and 
\begin{eqnarray*}
\pi _{x} &:&\nu _{x}^{1}\left( F\right) \longrightarrow N \\
\pi _{x} &:&v\mapsto \exp _{F}\left( \frac{\pi }{2}v\right)
\end{eqnarray*}%
are connected. We start with a review of Hebda's result.

\begin{definition}
(Hebda, \cite{Heda}) Consider geodesics $\gamma $ that leave $N$
orthogonally at time $0.$ $N$ has a very regular first focal locus if the
multiplicity of the first focal point is independent of $\gamma $ and, in
case the multiplicity is one, $\mathrm{\ker }\left( D\exp _{N}^{\perp
}\right) $ is contained in the tangent space to the tangent focal locus at
every first-tangent focal point.
\end{definition}

Along a geodesic that leaves our $N$ orthogonally at time $0,$ the
multiplicity of the focal point at time $\frac{\pi }{2}$ is 
\begin{eqnarray*}
\dim \mathcal{Z}_{F} &\mathcal{=}&\dim \left( \Lambda _{N}\right) -\dim 
\mathcal{T}_{F} \\
&=&\mathrm{\dim }\left( M\right) -1-\dim \left( F\right) ,
\end{eqnarray*}%
and hence is constant. Since the focal radius of $N$ along every geodesic is 
$\frac{\pi }{2},$ it follows from the Gauss Lemma that our $N$ has a very
regular first focal locus. Therefore, since $M$ is simply connected, we can
apply the following result of Hebda. (See Theorem 3.1 in \cite{Heda} and the
first line of its proof.)

\begin{theorem}
\label{Hebda 3.1}(Hebda, \cite{Heda}) Suppose $M$ is a connected, compact
Riemannian manifold, and $N$ is a connected, compact submanifold having a
very regular first focal locus such that the inclusion $\iota
:N\hookrightarrow M$ induces a surjection of fundamental groups.

If the multiplicity of the first focal points of $N$ is $s-1,$ then the
first focal locus $F$ of $N$ in $M$ is a submanifold of codimension $s$ that
coincides with the cut locus of $N$ in $M.$ Moreover, the tangent cut locus
of $N$ coincides with the first-tangent focal locus of $N,$ and $M$ is the
union of two disk bundles 
\begin{equation*}
M=D_{N}\cup _{\varphi }D_{F}
\end{equation*}%
over $N$ and $F$ respectively, where 
\begin{equation*}
\varphi :\partial D_{N}\longrightarrow \partial D_{F}
\end{equation*}%
is a diffeomorphism.
\end{theorem}

By combining transversality and Theorem \ref{Hebda 3.1} we get following.

\begin{corollary}
\label{pi_1 lemma}Suppose that $M$ is simply connected.

\noindent 1. If \textrm{codim}$\left( F\right) \geq 3,$ then $N$ is simply
connected.

\noindent 2. If \textrm{codim}$\left( N\right) \geq 3,$ then $F$ is simply
connected.
\end{corollary}

\begin{proof}
The two statements have dual proofs. We give the details for Part 1.

Transversality gives us 
\begin{equation*}
\pi _{1}\left( M\right) \cong \pi _{1}\left( M\setminus F\right) ,
\end{equation*}%
and by Theorem \ref{Hebda 3.1}, $M\setminus F$ deformation retracts to $N.$
\end{proof}

Similarly, the cut locus statements in Theorem \ref{Hebda 3.1} gives us

\begin{corollary}
\label{exp tilde prop}If $M$ is simply connected, then $\exp _{N}^{\perp }$
is injective on $B\left( N_{0},\frac{\pi }{2}\right) .$
\end{corollary}

\begin{lemma}
\label{connected fibers}Let $M$ be simply connected.

\noindent 1. If \textrm{codim}$\left( N\right) \geq 3,$ then the Riemannian
submersions 
\begin{equation*}
\nu _{x}^{1}\left( N\right) \longrightarrow F,\text{ }x\in N\text{ }
\end{equation*}%
have connected fibers with positive dimension.

\noindent 2. If \textrm{codim}$\left( F\right) \geq 3,$ then the Riemannian
submersions 
\begin{equation*}
\nu _{x}^{1}\left( F\right) \longrightarrow N,\text{ }x\in F\text{ }
\end{equation*}%
have connected fibers with positive dimension.
\end{lemma}

\begin{proof}
Since we have assumed that $\mathrm{dim}\left( F\right) +\mathrm{\dim }%
\left( N\right) \leq \dim \left( M\right) -2,$ 
\begin{eqnarray*}
\mathrm{\dim }\left( \nu _{x}^{1}\left( N\right) \right) &=&\dim \left(
M\right) -\mathrm{\dim }\left( N\right) -1 \\
&>&\mathrm{\dim }\left( F\right) .
\end{eqnarray*}%
Thus the fibers of $\nu _{x}^{1}\left( N\right) \longrightarrow F$ have
positive dimension.

By Corollary \ref{pi_1 lemma}, $F$ is simply connected if \textrm{codim}$%
\left( N\right) \geq 3.$ In this case, the long exact homotopy sequence for $%
\nu _{x}^{1}\left( N\right) \longrightarrow F,$ $x\in N$ gives 
\begin{equation*}
\pi _{1}\left( F\right) \longrightarrow \pi _{0}\left( \mathrm{fiber}\right)
\longrightarrow 0,
\end{equation*}%
since $\pi _{0}\left( \nu _{x}^{1}\left( N\right) \right) $ is trivial. Thus
the first conclusion holds. A similar argument gives us the second
conclusion, if \textrm{codim}$\left( F\right) \geq 3$.
\end{proof}

Since we have assumed that \addtocounter{algorithm}{1} 
\begin{equation}
\mathrm{dim}\left( F\right) +\mathrm{\dim }\left( N\right) \leq \dim \left(
M\right) -2,  \label{dim sum hyp}
\end{equation}%
\textrm{codim}$\left( N\right) \geq 2.$ Since $\mathrm{\dim }\left( N\right)
\geq 1,$ we have \textrm{codim}$\left( F\right) \geq 3.$ Combining this with
Lemma \ref{Focal F is pi/2 lemma}, Corollary \ref{exp tilde prop}, and Lemma %
\ref{connected fibers}, we have

\begin{theorem}
\label{to date thm}Let $M$ be a complete Riemannian $n$--manifold with $\Ric%
_{k}\geq k$ and $N$ any closed, connected, submanifold of $M$ with $\dim
\left( N\right) \geq k$ and focal radius $\frac{\pi }{2}.$ If $M$ is simply
connected and not isometric to the unit sphere, then:

\noindent 1. 
\begin{equation*}
\dim \left( N\right) +\dim \left( F\right) \leq n-2.
\end{equation*}

\noindent 2. $N$ is totally geodesic and isometric to an even dimensional
CROSS$.$

\noindent 3. The focal set $F$ of $N$ is totally geodesic and is either a
point or is isometric to an even dimensional CROSS.

\noindent 4. The normal exponential maps of $N$ and $F$ are injective on the 
$\frac{\pi }{2}$--balls around the zero sections of the normal bundles of $N$
and $F.$

\noindent 5. The conclusions of Proposition \ref{Jacobi Rigidity prop} hold
with $N$ replaced by $F.$

\noindent 6. For every $x\in F$ the map%
\begin{eqnarray*}
\pi _{x} &:&\nu _{x}^{1}\left( F\right) \longrightarrow N \\
\pi _{x} &:&v\mapsto \exp _{N}\left( \frac{\pi }{2}v\right)
\end{eqnarray*}%
is a Riemannian submersion whose fibers are connected and have positive
dimension.

\noindent 7. For every $x\in N$ the map%
\begin{eqnarray*}
\pi _{x} &:&\nu _{x}^{1}\left( N\right) \longrightarrow F \\
\pi _{x} &:&v\mapsto \exp _{N}\left( \frac{\pi }{2}v\right)
\end{eqnarray*}%
is a Riemannian submersion whose fibers are connected and have positive
dimension.
\end{theorem}

\section{Rigidity in the Sectional Curvature Case \label{M is not a sphere}}

In this section, we complete the proof of Theorem \ref{inter Ricci Rigidity
thm} in the case when the sectional curvature of $M$ is $\geq 1$. Since $M$
is simply connected, by Part 4 of Theorem \ref{to date thm}, the diameter of 
$M$ is $\geq \frac{\pi }{2}.$ So combining the Diameter Rigidity Theorem
with our dimension hypothesis (\ref{dim sum hyp}) and Theorem \ref{to date
thm} gives us the following.

\begin{proposition}
\label{codim 3 prop}If $M$ does not have constant curvature $1,$ then the
following hold.

\noindent 1. $M$ is isometric to a compact, rank one, symmetric space or is
homeomorphic to $S^{n}.$

\noindent 2. $N$ is even dimensional and isometric to the base of a Hopf
fibration.

\noindent 3. $F$ is either a point or is isometric to the base of a Hopf
fibration and is even dimensional
\end{proposition}

To conclude the proof of Theorem \ref{inter Ricci Rigidity thm}, we show
that conclusions 2 and 3 are not compatible with $M$ being a topological
sphere.

If $M$ is a sphere, the long exact homology sequence of the pair $\left(
M,F\right) $ gives 
\begin{equation*}
H_{q}\left( M,F\right) \cong H_{q-1}^{\#}\left( F\right)
\end{equation*}%
for $q\leq n-1.$

Use Theorem \ref{Hebda 3.1} to write 
\begin{equation*}
M=D_{N}\cup _{\varphi }D_{F}.
\end{equation*}%
By excision,%
\begin{equation*}
H_{q}\left( M,F\right) \cong H_{q}\left( D_{N},\partial D_{N}\right) .
\end{equation*}%
Thus for $q\leq n-1,$ \addtocounter{algorithm}{1} 
\begin{equation}
H_{q-1}^{\#}\left( F\right) \cong H_{q}\left( D_{N},\partial D_{N}\right) .
\label{hom eqn}
\end{equation}%
By Proposition \ref{codim 3 prop}, $F$ is either an even dimensional CROSS
or a point, and $N$ is an even dimensional CROSS. It follows from Equation (%
\ref{hom eqn}) that $H_{q}\left( D_{N},\partial D_{N}\right) \cong 0$ if $q$
is even and $\leq n-1.$ If $q$ is odd and $q+1\leq n-1$, then the sequence
of the pair $\left( D_{N},\partial D_{N}\right) $ gives 
\begin{equation*}
0=H_{q+1}\left( D_{N},\partial D_{N}\right) \longrightarrow H_{q}\left(
\partial D_{N}\right) \longrightarrow H_{q}\left( N\right) =0,
\end{equation*}%
since $N$ is an even dimensional CROSS. Thus \addtocounter{algorithm}{1} 
\begin{equation}
H_{q}\left( \partial D_{N}\right) \cong 0\text{ if }q\text{ is odd and }\leq
n-2.  \label{odd hom vanish}
\end{equation}%
Since $N$ is isometric to the base of a Hopf fibration with connected
fibers, $\dim \left( N\right) \geq 2.$ Since $\dim \left( N\right) +\dim
\left( F\right) \leq n-2,$ we get $n\geq 4.$ So $\dim \left( \partial
D_{N}\right) \geq 3.$ Since $\partial D_{N}$ is a connected, compact, odd
dimensional manifold, \eqref{odd hom vanish} implies, via Poincar\'{e}
duality, that $\partial D_{N}$ is a $\mathbb{Z}_{2}$--homology sphere of
dimension $\geq 3$ . Thus the Mayer--Vietoris sequence with $q\in \left\{
2,\ldots ,n-2\right\} $ yields%
\begin{equation*}
0=H_{q}\left( \partial D_{N}\right) \longrightarrow H_{q}\left( D_{N}\right)
\oplus H_{q}\left( D_{F}\right) \longrightarrow H_{q}\left( M\right)
\longrightarrow H_{q}\left( \partial D_{N}\right) =0.
\end{equation*}%
Since $D_{N}$ has the homotopy type of the CROSS $N,$ and $\dim \left(
M\right) \geq \dim \left( N\right) +2\geq 4,$ $M$ cannot be homeomorphic to
a sphere.

\section{Rigidity and Intermediate Ricci\label{Rigid and Inter ric sect}}

In this section, we complete the proof of Theorem \ref{inter Ricci Rigidity
thm}. This is achieved by analyzing the radial geometry from $N$ and $F.$
Proposition \ref{Jacobi Rigidity prop}, Lemma \ref{duality lemma copy(1)},
and Lemma \ref{Focal F is pi/2 lemma} give us rigid radial geometry along
the distribution spanned by the Jacobi fields in $\mathcal{T}_{N}$ and $%
\mathcal{T}_{F}.$ To prove rigidity for the $\mathcal{Z}$--Jacobi fields, we
show, in Proposition \ref{Jacobi stru prop}, that as in the proof of the
Diameter Rigidity Theorem, there are enough other dual pairs in $M$ to force
the $\mathcal{Z}$--Jacobi fields to span projective lines. This is achieved
via the next three results, wherein the hypotheses that $N$ is connected and 
$M$ is simply connected are still in force.

\begin{proposition}
\label{inj p prop}\noindent 1. For any $p\in N$ the cut point along any
geodesic emanating from $p$ is at distance $\frac{\pi }{2}$ from $p.$

\noindent 2. For any $p,q\in N,$ any minimal geodesic of $M$ between $p$ and 
$q$ lies entirely in $N.$
\end{proposition}

\begin{proof}
Let $v\in T_{p}M\setminus \left\{ T_{p}N,\nu _{p}\left( N\right) \right\} $
be any unit vector. Let $v^{T}$ and $v^{\perp }$ be the unit vectors that
point in the same directions as the projections of $v$ onto $T_{p}N$ and $%
\nu _{p}\left( N\right) ,$ respectively. By Part 4 of Proposition \ref%
{Jacobi Rigidity prop}, \textrm{span}$\left\{ v^{T},v^{\perp }\right\} $
exponentiates to a totally geodesic immersed surface $\Sigma $ of constant
curvature $1$ that contains $\gamma _{v}.$ By Corollary \ref{exp tilde prop}%
, the restriction of $\exp _{p}$ to the interior of the circular sector 
\begin{equation*}
\mathrm{Sect}\left( v^{T},v^{\perp },\frac{\pi }{2}\right) \equiv \left\{
\left. \exp _{p}\left( t\left( v^{T}\cos s+v^{\perp }\sin s\right) \right) 
\text{ }\right\vert \text{ }t,s\in \left[ 0,\frac{\pi }{2}\right] \right\}
\end{equation*}%
of radius and angle $\frac{\pi }{2}$ spanned by $v^{T}$ and $v^{\perp }$ is
an embedding. If $w\in T_{p}M$ is not in $T_{p}N,\nu _{p}\left( N\right) ,$
or $\mathrm{span}\left\{ v^{T},v^{\perp }\right\} ,$ then Corollary \ref{exp
tilde prop} gives that interiors of $\mathrm{Sect}\left( v^{T},v^{\perp },%
\frac{\pi }{2}\right) $ and $\mathrm{Sect}\left( w^{T},w^{\perp },\frac{\pi 
}{2}\right) $ are disjoint. It follows that the cut-time of any unit $v\in
T_{p}M\setminus \left\{ T_{p}N,\nu _{p}\left( N\right) \right\} $ is $\geq 
\frac{\pi }{2},$ and by continuity, the cut-time of any unit $v\in T_{p}M$
is $\geq \frac{\pi }{2}.$

Since the focal radius of $N$ is $\frac{\pi }{2},$ to complete the proof of
Part 1, it suffices to check that any unit $v\in T_{p}M\setminus \nu
_{p}\left( N\right) $ has a cut point at time $\frac{\pi }{2}.$ Since $N$ is
a CROSS with curvature in $\left[ 1,4\right] ,$ there is a unit vector $%
w^{T}\in T_{p}\left( N\right) $ with%
\begin{equation*}
\gamma _{w^{T}}\left( \frac{\pi }{2}\right) =\gamma _{v^{T}}\left( \frac{\pi 
}{2}\right) \text{ and }w^{T}\neq v^{T}.
\end{equation*}

Let $u\in \nu _{\gamma _{v^{T}}\left( \frac{\pi }{2}\right) }\left( N\right) 
$ be a unit vector. Let $U_{v}$ and $U_{w}$ be the backwards parallel
transports of $u$ along $\gamma _{v^{T}}$ and $\gamma _{w^{T}}.$ Let $\Sigma
_{v^{T}}$ and $\Sigma _{w^{T}}$ be the spherical sectors of radius and angle 
$\frac{\pi }{2}$ obtained via Part 4 of Proposition \ref{Jacobi Rigidity
prop} by exponentiating $U_{v}$ and $U_{w}.$ That is%
\begin{equation*}
\Sigma _{v^{T}}\equiv \left\{ \left. \exp _{N}\left( tU_{v}\left( s\right)
\right) \text{ }\right\vert \text{ }s,t\in \left[ 0,\frac{\pi }{2}\right]
\right\}
\end{equation*}%
and 
\begin{equation*}
\Sigma _{w^{T}}\equiv \left\{ \left. \exp _{N}\left( tU_{w}\left( s\right)
\right) \text{ }\right\vert \text{ }s,t\in \left[ 0,\frac{\pi }{2}\right]
\right\} .
\end{equation*}

Then $\Sigma _{w^{T}}$ and $\Sigma _{v^{T}}$ are different surfaces, but
since 
\begin{equation*}
U_{v}\left( \frac{\pi }{2}\right) =U_{w}\left( \frac{\pi }{2}\right) =u\in
\nu _{\gamma _{v^{T}}\left( \frac{\pi }{2}\right) }\left( N\right) ,
\end{equation*}%
$\Sigma _{w^{T}}$ and $\Sigma _{v^{T}}$ intersect along $\exp _{N}\left(
tu\right) .$ So every cut point from $p$ occurs at distance $\frac{\pi }{2}$
from $p.$

By Corollary \ref{exp tilde prop}, the intersection of ~$N$ with any of the
sectors $\mathrm{Sect}\left( v^{T},v^{\perp },\frac{\pi }{2}\right) $ is
precisely $\gamma _{v^{T}}\left[ 0,\frac{\pi }{2}\right] \subset N.$ Part 2
follows.
\end{proof}

\begin{proposition}
\label{A(p) prop}For any $p\in N,$ the set of points in $M$ at distance $%
\frac{\pi }{2}$ from $p,$ 
\begin{equation*}
A\left( p\right) \equiv S\left( p,\frac{\pi }{2}\right) ,
\end{equation*}%
is a closed submanifold of dimension 
\begin{equation*}
\dim \left( N\right) +\dim F
\end{equation*}%
and focal radius $\frac{\pi }{2}$.
\end{proposition}

\begin{proof}
Let 
\begin{equation*}
A_{N}\left( p\right) \equiv \left\{ \left. x\in N\text{ }\right\vert \text{ 
\textrm{dist}}\left( p,x\right) =\frac{\pi }{2}\right\} .
\end{equation*}%
Using the rigid hinges of Part 4 of Proposition \ref{Jacobi Rigidity prop}
and the fact that $A\left( p\right) \equiv S\left( p,\frac{\pi }{2}\right) $
is the cut locus of $p$ we see that we can describe $A\left( p\right) $ in
two ways: 
\begin{equation*}
A\left( p\right) =\left\{ \left. \exp _{N}^{\perp }\left( tv\right) \text{ }%
\right\vert \text{ }v\in \nu ^{1}\left( N\right) |_{A_{N}\left( p\right) }%
\text{ and }t\in \left[ 0,\frac{\pi }{2}\right] \right\} ,
\end{equation*}%
and \addtocounter{algorithm}{1} 
\begin{equation}
A\left( p\right) =\left\{ \left. \exp _{F}^{\perp }\left( tv\right) \text{ }%
\right\vert \text{ }v\in \nu ^{1}\left( F\right) \text{, }\exp _{F}\left( 
\frac{\pi }{2}v\right) \in A_{N}\left( p\right) ,\text{ and }t\in \left[ 0,%
\frac{\pi }{2}\right] \right\} .  \label{2nd descr}
\end{equation}%
Either description shows $A\left( p\right) $ is compact. Since both $N$ and $%
F$ have focal radius $\frac{\pi }{2},$ both descriptions show, via Corollary %
\ref{exp tilde prop}, that $A\left( p\right) \setminus \left\{ F\cup
N\right\} $ is a manifold. The first description shows that $A\left(
p\right) $ is smooth near $N$.

Recall that for every $x\in F,$ the map%
\begin{eqnarray*}
\pi _{x} &:&\nu _{x}^{1}\left( F\right) \longrightarrow N, \\
\pi _{x} &:&v\mapsto \exp _{N}\left( \frac{\pi }{2}v\right)
\end{eqnarray*}%
is a Riemannian submersion with connected fibers. Combining this with Part 4
of Proposition \ref{Jacobi Rigidity prop} and Proposition \ref{inj p prop},
we rewrite \eqref{2nd descr} as \addtocounter{algorithm}{1} 
\begin{equation}
A\left( p\right) =\cup _{x\in F}\left\{ \left. \exp _{F}^{\perp }\left(
tv\right) \text{ }\right\vert \text{ }v\in \nu _{x}^{1}\left( F\right) \text{%
, }v\perp \pi _{x}^{-1}\left( p\right) ,\text{ and }t\in \left[ 0,\frac{\pi 
}{2}\right] \right\} ,  \label{A(p) 2 eqn}
\end{equation}%
where $\left\{ v\right\} $ and $\pi _{x}^{-1}\left( p\right) $ are subsets
of $\nu _{x}\left( F\right) ,$ and the notion of perpendicular comes from
the inner product that $g$ induces on $\nu _{x}\left( F\right) .$ This shows 
$A\left( p\right) $ is smooth near $F.$

Next, we decompose $\nu _{x}^{1}\left( F\right) $ as a join 
\begin{eqnarray*}
\nu _{x}^{1}\left( F\right) &=&\pi _{x}^{-1}\left( p\right) \ast \left( \pi
_{x}^{-1}\left( p\right) \right) ^{\perp },\text{ where } \\
\left( \pi _{x}^{-1}\left( p\right) \right) ^{\perp } &\equiv &\left\{
\left. v\in \nu _{x}^{1}\left( F\right) \text{ }\right\vert \text{ }v\perp
\pi _{x}^{-1}\left( p\right) \right\} .\text{ }
\end{eqnarray*}%
Note that \eqref{A(p) 2 eqn} gives that for any $x\in F$, %
\addtocounter{algorithm}{1} 
\begin{equation}
\dim \left( A\left( p\right) \right) =\dim F+\dim \left( \pi _{x}^{-1}\left(
p\right) \right) ^{\perp }+1.  \label{dim A-p 1}
\end{equation}%
Since $\pi _{x}:\nu _{x}^{1}\left( F\right) \longrightarrow N$ is a
Riemannian submersion, \addtocounter{algorithm}{1} 
\begin{equation}
\dim \left( N\right) =\dim \left( H_{v}\left( \pi _{x}^{-1}\left( p\right)
\right) \right) ,  \label{dim N}
\end{equation}%
where $H_{v}\left( \pi _{x}^{-1}\left( p\right) \right) $ is the horizontal
space for $\pi _{x}$ at any $v\in \pi _{x}^{-1}\left( p\right) \subset \nu
_{x}^{1}\left( F\right) .$

The join decomposition $\nu _{x}^{1}\left( F\right) =\pi _{x}^{-1}\left(
p\right) \ast \left( \pi _{x}^{-1}\left( p\right) \right) ^{\perp }$
identifies $\left( \pi _{x}^{-1}\left( p\right) \right) ^{\perp }$ with the
unit vectors in $H_{v}\left( \pi _{x}^{-1}\left( p\right) \right) .$ So%
\begin{equation*}
\dim \left( \pi _{x}^{-1}\left( p\right) \right) ^{\perp }=\dim \left(
H_{v}\left( \pi _{x}^{-1}\left( p\right) \right) \right) -1.
\end{equation*}%
Combining with Equations \eqref{dim A-p 1} and \eqref{dim N}, we get%
\begin{equation*}
\dim \left( A\left( p\right) \right) =\dim F+\dim \left( N\right) -1+1.
\end{equation*}

Since $A\left( p\right) =S\left( p,\frac{\pi }{2}\right) $ is a smooth
submanifold, and every geodesic that leaves $p$ has cut point at distance $%
\frac{\pi }{2}$ from $p,$ it follows from $1^{st}$--variation that every
geodesic leaving $p$ arrives orthogonally at $A\left( p\right) $ at time $%
\frac{\pi }{2}.$ This identifies the unit tangent sphere at $p,$ $S_{p},$
with the unit normal bundle of $A\left( p\right) ,$ $\nu ^{1}\left( A\left(
p\right) \right) .$ Combined with Proposition \ref{inj p prop}, it follows
that the focal radius of $A\left( p\right) $ along any normal geodesic is
greater than or equal to $\frac{\pi }{2}$ . So it follows from Theorem \ref%
{Intermeadiate Ricci Thm} that the focal radius of $A\left( p\right) $ is
exactly $\frac{\pi }{2}.$
\end{proof}

It follows that Theorem \ref{to date thm} applies with $N^{p}=A\left(
p\right) $ and $F^{p}=p.$

Next, we apply Propositions \ref{inj p prop} and \ref{A(p) prop} to $q\in
A\left( p\right) $ and, with a further application of Theorem \ref{to date
thm}, get the following result.

\begin{proposition}
\label{Cross again}\noindent 1. Every cut point from $q$ occurs at distance $%
\frac{\pi }{2}$ from $q.$

\noindent 2. The set of points in $M$ at distance $\frac{\pi }{2}$ from $q,$ 
\begin{equation*}
A\left( q\right) \equiv S\left( q,\frac{\pi }{2}\right) ,
\end{equation*}%
is a closed submanifold with focal radius $\frac{\pi }{2}$ and dimension $%
\dim N+\dim F.$

\noindent 3. $A\left( q\right) $ is totally geodesic and isometric to a
CROSS.

\noindent 4. For any $a,b\in A\left( q\right) ,$ any minimal geodesic of $M$
between $a$ and $b$ lies entirely in $A\left( q\right) .$
\end{proposition}

Returning $N^{p}=A\left( p\right) $ and $F^{p}=p,$ we now have the following
refinement of Proposition \ref{Jacobi Rigidity prop}.

\begin{proposition}
\label{Jacobi stru prop}Let $\gamma $ be a unit speed geodesic that leaves $%
N^{p}$ orthogonally at time $0$ and let $\Lambda _{N^{p}},$ $\mathcal{T}%
_{N^{p}}$ and $\mathcal{Z}_{N^{p}}$ be as in (\ref{Lambda dfn}).

\noindent 1. $\mathcal{Z}_{N^{p}}\oplus \mathcal{T}_{N^{p}}$ is a parallel,
orthogonal splitting of $\Lambda _{N^{p}}$ along $\left( 0,\frac{\pi }{2}%
\right) .$

\noindent 2. $\mathcal{T}_{N^{p}}$ and $\mathcal{Z}_{N^{p}}$ have the forms 
\begin{eqnarray*}
\mathcal{T}_{N^{p}} &\equiv &\left\{ \cos tE|\text{ }E\text{ is parallel and
tangent to }N^{p}\text{ at time }0\right\} , \\
\mathcal{Z}_{N^{p}} &\equiv &\left\{ \frac{1}{2}\sin 2tE|\text{ }E\text{ is
parallel and orthogonal to }N^{p}\text{ at time }0\right\} .
\end{eqnarray*}
\end{proposition}

\begin{proof}
Apart from the second equation in Part 2, this is a repeat of Parts 1 and 2
of Proposition \ref{Jacobi Rigidity prop}. The second equation in Part 2
follows from Part 2 of Proposition \ref{Cross again}. Indeed, let $\gamma $
be a unit speed geodesic from $x\in N^{p}$ to $p.$ Choose $q\in N^{p}$ at
distance $\frac{\pi }{2}$ from $x,$ and apply Proposition \ref{Cross again}
to $q.$ It follows that $\gamma \subset A\left( q\right) ,$ and $A\left(
q\right) $ is a totally geodesic CROSS. In particular, the Jacobi fields
along $\gamma $ have the indicated form if they are tangent to $A\left(
q\right) .$ Since the normal space to $A\left( q\right) $ along $\gamma $ is
spanned by a subspace of $\mathcal{T}_{N^{p}},$ the result follows.
\end{proof}

We finish the proof of Theorem \ref{inter Ricci Rigidity thm} along the
lines of the proof of Theorem 4.3 in \cite{GrovGrom} by using Cartan's
Theorem (Theorem 2.1, page 157 of \cite{doCarm}).

Since Theorem \ref{to date thm} applies to $N^{p}=A\left( p\right) ,$ there
is a CROSS $P$ with 
\begin{equation*}
\dim \left( P\right) =\dim \left( M\right)
\end{equation*}%
and 
\begin{equation*}
\mathrm{\dim }\left( \mathcal{Z}_{N^{p}}\right) =\dim \left( \mathbb{F}%
\right) -1,
\end{equation*}%
where $\mathbb{F}$ is the division algebra that defines $P.$

Choose a point $\tilde{p}\in P.$ Since $P$ is a CROSS, we have a Riemannian
submersion 
\begin{equation*}
\tilde{\pi}_{\tilde{p}}:S_{\tilde{p}}\longrightarrow A\left( \tilde{p}%
\right) \equiv \left\{ \left. x\in P\text{ }\right\vert \text{ \textrm{dist}}%
\left( x,P\right) =\frac{\pi }{2}\right\}
\end{equation*}%
that is isometrically equivalent to a Hopf Fibration. Since $\mathrm{\dim }%
\left( \mathcal{Z}_{N^{p}}\right) =\dim \left( \mathbb{F}\right) -1,$ we
have, using \cite{GrvGrm1} and \cite{Wilk}, that $\tilde{\pi}_{\tilde{p}}$
is isometrically equivalent to%
\begin{equation*}
\pi _{p}:S_{p}\longrightarrow N^{p}\equiv A\left( p\right) .
\end{equation*}%
Let 
\begin{equation*}
I:S_{\tilde{p}}\longrightarrow S_{p}
\end{equation*}%
be a linear isometric equivalence between $\tilde{\pi}_{\tilde{p}}$ and $\pi
_{p}.$ Then we have a commutative diagram%
\begin{equation*}
\begin{array}{lll}
S_{\tilde{p}} & \overset{I}{\longrightarrow } & S_{p} \\ 
\downarrow \tilde{\pi}_{\tilde{p}} &  & \downarrow \pi _{p} \\ 
A\left( \tilde{p}\right) & \overset{\hat{I}}{\longrightarrow } & A\left(
p\right)%
\end{array}%
\end{equation*}%
Since $\tilde{\pi}_{\tilde{p}}$ and $\pi _{p}$ are Riemannian submersions
and $I$ is an isometry, $\hat{I}$ is an isometry.

Via the Cartan Theorem and Proposition \ref{Jacobi stru prop}, we see that $%
\iota \equiv \exp _{p}\circ I\circ \exp _{\tilde{p}}^{-1}$ defines an
isometry between $P\setminus $ $A\left( \tilde{p}\right) $ and $M\setminus
A\left( p\right) $ that induces the isometry $\hat{I}:A\left( \tilde{p}%
\right) \longrightarrow A\left( p\right) ,$ and thus $\iota $ extends to an
isometry $P\longrightarrow M.$ This completes the proof of Theorem \ref%
{inter Ricci Rigidity thm}.

\end{document}